\documentclass[12pt]{amsart}
\usepackage{amssymb}
\usepackage{amsmath}
\usepackage{setspace}
\usepackage{mathtools}
\usepackage{verbatim}
\usepackage{lipsum}
\usepackage{amsthm}
\usepackage{tikz-cd}
\usepackage{hyperref}

\newtheorem{theorem}{Theorem}[section]
\newtheorem{lemma}[theorem]{Lemma}
\newtheorem{corollary}[theorem]{Corollary}

\newtheorem{remark}[theorem]{Remark}

\newtheorem{proposition}[theorem]{Proposition}
\newtheorem{propdef}[theorem]{Proposition-Definition}
\newtheorem{construction}[theorem]{Construction}

\newtheorem{example}[theorem]{Example}
\newtheorem*{claim}{Claim}

\newcommand{\bb}[1]{\mathbb{#1}}
\DeclareMathOperator{\sing}{Sing}

\newcommand{\cal}[1]{\mathcal{#1}}

\newcommand{\Spec}[1]{\text{Spec}#1}
\title{Foliation adjunction}

\author{Paolo Cascini}
\address{Department of Mathematics\\
Imperial College London\\
180 Queen's Gate\\
London SW7 2AZ, UK}
\email{p.cascini@imperial.ac.uk}

\author{Calum Spicer}
\address{Department of Mathematics, King's College London, Strand,
London WC2R 2LS, UK}
\email{calum.spicer@kcl.ac.uk}

\subjclass[2010]{14E30, 37F75}

\begin{document}

\setcounter{tocdepth}{1}

\begin{abstract}
We present an adjunction formula for foliations on varieties and we consider applications
of the adjunction formula to the cone theorem for rank one foliations and 
the study of foliation singularities. 
\end{abstract}
\maketitle

\tableofcontents

\section{Introduction}
Our primary goal is to present an adjunction formula for foliations on 
varieties defined over a field of characteristic zero.

Various special cases of the adjunction formula have appeared in
\cite{Brunella00}, \cite{McQuillan08}, \cite{AD13}, \cite{AD14}, 
\cite{Spicer20}, \cite{CS18}, \cite{CS20} and \cite{ACSS}. 
We will give a treatment of the adjunction formula which
unifies these other cases, and which is in line with the treatment
of the adjunction formula for varieties (see for instance 
\cite{Kollar13}). 

Recall that 
the adjunction formula for varieties relates the canonical divisor of a smooth variety $X$
and the canonical divisor of a smooth codimension one subvariety $D \subset X$
by the linear equivalence \[(K_X+D)\vert_D \sim K_D.\]

Given a foliation $\mathcal F$ of rank $r$ on a smooth variety $X$ and a smooth codimension one
subvariety $D \subset X$ we can define a restricted foliation $\mathcal F_D$ on $D$.
Roughly speaking, the leaves of $\mathcal F_D$ are the leaves of $\mathcal F$ intersected
with $D$.  Thus, $\mathcal F_D$ is a foliation of rank $r-\epsilon(D)$ where 
$\epsilon(D) = 1$ if $D$ is not invariant by $\mathcal F$
and $\epsilon(D) = 0$ otherwise.
The adjunction formula for foliations relates the divisors $(K_{\mathcal F}+\epsilon(D)D)\vert_D$
and $K_{\mathcal F_D}$.  Even if $X$, $D$ and $\mathcal F$ are all smooth, these two divisors
will not in general be linearly equivalent. In analogy with the adjunction formula for singular
varieties (e.g. see \cite[\S 16]{Kollaretal}), we introduce a correction term ${\rm Diff}_D(\mathcal F) \geq 0$, 
called the different.  With this correction term, we have a linear equivalence
\[(K_{\mathcal F}+\epsilon(D)D)\vert_D \sim K_{\mathcal F_D}+{\rm Diff}_D(\mathcal F).\]
We refer to Proposition-Definition \ref{prop_adjunction} for the definition and construction
of the different. Note that a similar statement appeared in \cite[Definition 3.11]{AD14} in the case that $D$ is invariant by $\cal F$. 

In the case where $X$ is a smooth surface and $C$ is a smooth curve, this 
formula has been discussed
in \cite{Brunella00} in slightly different terms.
When $C$ is invariant by $\mathcal F$ the restricted foliation is the foliation with one leaf
and so $K_{\mathcal F_C} = K_C$.
In this case we have by \cite[Proposition 2.3]{Brunella00}
\[K_{\mathcal F}\vert_C \sim K_{\mathcal F_C}+Z(\mathcal F, C)\] where
$Z(\mathcal F, C)$ depends on the singularities of $\mathcal F$ along $C$.
When $C$ is not invariant by $\mathcal F$ the restricted foliation is the foliation by
points and so $K_{\mathcal F_C} = 0$. 
In this case we have by \cite[Proposition 2.2]{Brunella00}
\[(K_{\mathcal F}+C)\vert_C \sim K_{\mathcal F_C}+{\rm tang}(\mathcal F, C)\]
where ${\rm tang}(\mathcal F, C)$ is the tangency divisor between $\mathcal F$ and $C$.

We will see that the correction term, $Z(\mathcal F, C)$ or ${\rm tang}(\mathcal F, C)$,
in each of these cases is equal to the different 
${\rm Diff}_C(\mathcal F)$.  
Assuming that $X$ and $D$ are possibly singular and that $S\to D$ is the normalisation map, 
we will explain the construction of the restricted foliation $\mathcal F_S$ on $S$ and the induced different,
 see Proposition-Definition \ref{prop_adjunction}.  
We will also calculate the different in many cases, see 
Proposition \ref{prop_explicit_calc}.

\medskip

Another important feature of the adjunction formula is that it gives a
way to relate the singularities of $\mathcal F$ to $\mathcal F_D$. As an example
of this type of statement we have the following:

\begin{theorem}[cf. Theorem \ref{thm_adj_sing}]
Let $X$ be a $\mathbb Q$-factorial variety, 
let $\mathcal F$ be a foliation of rank $r$ and let $D$ be a prime divisor which is not 
$\mathcal F$-invariant.
Suppose that $(\mathcal F, D)$
is canonical (resp. log canonical).  
Let $S \rightarrow D$ be the normalisation of $D$.  

Then $(\mathcal F_S, {\rm Diff}_S(\cal F))$
is canonical (resp. log canonical).
\end{theorem}

As examples show, see Section \ref{s_failure}, 
the non-invariance of $D$ in the above statement is necessary.

\medskip

Finally, we turn to some applications of the adjunction formula.
Our first application is a proof of the cone theorem for pairs $(\mathcal F, \Delta)$
where $\mathcal F$ is a rank one foliation
(see also 
\cite[Corollary IV.2.1]{bm16} and \cite{McQ05}):

\begin{theorem}[= Theorem \ref{thm_cone}]
Let $X$ be a normal projective variety, let $\cal F$ be a rank one foliation 
	and let $\Delta \geq 0$ so that $K_{\mathcal F}$ and $\Delta$ are $\mathbb Q$-Cartier and 
$(\mathcal F, \Delta)$ is log canonical.

Then there are $\mathcal F$-invariant rational curves $C_1,C_2,\dots$ such that 
\[
0<-(K_{\cal F}+\Delta)\cdot C_i\le 2\dim X\]
and
$$\overline{\rm NE}(X)=\overline{\rm NE}(X)_{K_{\cal F}+\Delta\ge 0}+
\sum_i \mathbb R_+[C_i].$$
\end{theorem}

In \cite[Theorem 1.2]{CouPer06} a dynamical characterisation of ample line bundles on smooth 
surfaces was provided.
As a consequence of the above theorem we are able to extend this to higher dimensions.

\begin{theorem}[= Corollary \ref{cor_ampleness}]
Let $X$ be a normal projective variety and let $L$ be a $\mathbb Q$-Cartier divisor. 
Suppose that
\begin{enumerate}
\item $L^{\dim X} \neq 0$; 

\item for some $q \in \mathbb Q_{>0}$ there exists a rank one foliation $\mathcal F$ with 
$K_{\mathcal F} \equiv qL$; and 

\item $\mathcal F$ has isolated singularities and $\mathcal F$ admits no
invariant positive dimensional subvarieties.
\end{enumerate}
Then $L$ is ample.
\end{theorem}

In fact, \cite{CouPer06} proves a converse to the above theorem: if $L$ is an ample $\mathbb Q$-Cartier divisor, $n \gg 0$ is sufficiently
divisible and 
$\mathcal F$ is a general foliation with $K_{\mathcal F} \equiv nL$, then $\mathcal F$ admits no invariant positive dimensional subvarieties
and has isolated singularities.

%

Furthermore, we consider 
the study of singularities of foliations with a non-trivial algebraic part:

\begin{theorem}[cf. Theorem \ref{thm_alg_almost_hol}]
Let $X$ be a $\mathbb Q$-factorial klt projective variety and let $\mathcal F$ be a foliation with canonical singularities.

Then the algebraic part of $\mathcal F$ is induced by an almost holomorphic map.
\end{theorem}

We expect the Minimal Model Program
to have interesting implications for the study of foliation singularities.
Indeed, following \cite[Theorem 1.6]{CS18} and \cite[Lemma 2.8]{CS20}, we believe that there is a close
relation between the classes of singularities of the Minimal Model Program
and the dicriticality properties of the foliation.  In particular, we expect 
that canonical singularities satisfy some suitable non-dicritical condition.
Theorem \ref{thm_alg_almost_hol} is a partial confirmation of this 
in the case of foliations with non-trivial algebraic part (see \cite[Conjecture 4.2]{CS23} in the case of algebraically integrable foliations). 

Building off of work of \cite{Ou21} this theorem has implications for the study 
of foliations where $-K_{\mathcal F}$ is nef.

\begin{corollary}[cf. Corollary \ref{cor_main}]
Let $X$ be a smooth projective variety and let $\mathcal F$ be a 
foliation with canonical 
singularities.
Suppose that $-K_{\mathcal F}$ is nef and is not numerically trivial.

Then the algebraic part of $\mathcal F$ is induced by an equidimensional fibration.
\end{corollary}

\subsection{Acknowledgements} We would like to thank Mengchu Li, Jihao Liu, James M\textsuperscript cKernan and Stefania Vassiliadis for many useful discussions. 
We are grateful to the referee for carefully reading the paper and for several useful suggestions and corrections.
Both the authors are partially funded by EPSRC.

\section{Preliminaries}
All our schemes
 are Noetherian, pure dimensional 
and separated over an uncountable algebraically closed field $K$ of characteristic zero, unless otherwise specified.
The results here hold equally well for algebraic spaces
and for complex analytic varieties. We will use the fact that many results in the Minimal Model Program for algebraic varieties have been recently generalised to  projective morphisms between complex analytic spaces
 (see \cite{DHP24, Fujino22, LM22}).

\subsection{Line and divisorial sheaves}
Let $X$ be a
not necessarily reduced $S_2$ scheme.  A {\bf line sheaf} $L$ on $X$ is 
a coherent rank one $S_2$ sheaf such that there exists a closed subscheme 
$Z \subset X$ of codimension at least two
such that $L\vert_{X \setminus Z}$ is locally free. 

We recall that if $E$ is a coherent $S_2$ sheaf, $Z \subset X$ is a closed subscheme of 
codimension at least two and $i\colon X \setminus Z \rightarrow X$ is the inclusion 
then the natural morphism $E \rightarrow i_*E\vert_{X \setminus Z}$ is an isomorphism 
\cite[Chapter III, Exercise 3.5]{Hartshorne77}. 
Moreover, by \cite[Proposition 51.8.7]{stacks-project} if $Z \subset X$ is a closed subscheme of codimension at least two, 
$i\colon X \setminus Z \rightarrow X$ is the inclusion  
and   $L$ is a locally free sheaf on $X \setminus Z$ then $i_*L$ is a coherent sheaf.  
It is immediate that $i_*L$ is an $S_2$ sheaf. Given an  integer $n$, we may define the line sheaf $L^{[n]}$ to be $i_*(L\vert_{X\setminus Z}^{\otimes n})$
where $i\colon X\setminus Z \rightarrow X$ is the inclusion.

We denote by  ${\rm LSh}(X)$ the group of isomorphism classes of such sheaves
and we define ${\rm LSh}(X)_{\mathbb Q} \coloneqq {\rm LSh}(X)\otimes \mathbb Q$
to be the group of $\mathbb Q$-line sheaves.

If $X$ is reduced, a {\bf divisorial sheaf} on $X$ is the data of a line sheaf $L$ together with a choice
of an embedding $L \rightarrow K(X)$ 
and we denote the group of isomorphism classes of such sheaves by ${\rm WSh}(X)$.  
We likewise define the group of $\mathbb Q$-divisorial sheaves ${\rm WSh}(X)_{\mathbb Q}$.

Consider a scheme $X$ and a $\mathbb Q$-line sheaf $L$.  Let $Z \subset X$ be a codimension
two subscheme and let $n>0$ be an integer such that $L^{[ n]}\vert_{X\setminus Z}$
is locally free on $X \setminus Z$.
Let $D$ be an $S_2$ scheme and let $f\colon D \rightarrow X$ be a morphism such that
$f^{-1}(Z) \subset D$ is of codimension at least two.
We define the divisorial pullback $f^wL$ to be $\frac{1}{n}j_*(f^*(L^{[n]}\vert_{X\setminus Z}))$
where $j\colon D \setminus f^{-1}(Z) \rightarrow D$ is the inclusion. Note that $f^\omega L$ is a $\mathbb Q$-line sheaf on $D$.
We can likewise define the restriction (and pullback) of a divisorial sheaf.
If $L$ is the sheaf defined by a $\mathbb Q$-Cartier divisor, then these notions all agree with the restriction and pullback
of $\mathbb Q$-Cartier divisors. Similarly, if $G$ is a prime divisor on $X$ such that $mG$ is Cartier on $X\setminus Z$ for some positive integer $m$ and no component of  $f(D)$ is  contained in the support of $G$, then we define $f^wG$ to be $\frac 1 m G'$ where $G'$ is the closure of $f^*(mG|_{X\setminus Z})$ in $D$. By linearity, we can extend the definition to any $\mathbb Q$-divisor on $X$ which is $\mathbb Q$-Cartier on $X\setminus Z$.

\subsection{Integrable distributions and foliations}

Let $X$ be a not necessarily reduced $S_2$ scheme.
A rank $r$ {\bf integrable distribution} $\mathcal F$ on $X$ is the data of a 
\begin{enumerate}
\item a line sheaf $L$; and
\item a Pfaff field, i.e., a morphism $\phi\colon \Omega^r_X \rightarrow L$,
satisfying the following integrability condition:
in some neighbourhood $U$
of the generic point of each irreducible component of 
$X$ there exists a coherent rank $r$ sheaf $E$
and a surjective morphism $q\colon \Omega^1_U \rightarrow E$ such that the $r$-th wedge
power of this morphism agrees with $\phi|_U$ 
and its dual $E^*\hookrightarrow T_U$ is closed under Lie bracket. 
\end{enumerate}

We define the {\bf canonical class} of the integrable distribution $\cal F$ to be any Weil divisor
$K_{\mathcal F}$ on $X$
such that $\mathcal O_X(K_{\mathcal F}) \cong L$.
  A rank $r$ {\bf foliation} on $X$ is a rank $r$ integrable distribution on $X$ whose Pfaff field $\phi$ is such that ${\rm coker}~ \phi$ 
is supported in codimension at least two. 
 Given a rank $r$ integrable distribution $\mathcal F$ on a normal scheme $X$ 
we define the {\bf singular locus} of $\cal F$, denoted $\sing \, {\mathcal F}$, to be the co-support
of the ideal sheaf defined by the image of  the induced map 
$(\Omega^r_X\otimes \mathcal O_X(-K_{\mathcal F}))^{**} \rightarrow \mathcal O_X$.

For foliations on normal varieties, this definition agrees with the usual definition, see
Section \ref{s_normal} below.
Elsewhere in the literature there are differing definitions for foliations
on general schemes, we refer to Section \ref{s_otherdef} for a discussion of this point.
 
 \medskip

Let $X$ be a not necessarily reduced $S_2$ scheme.
We define a {\bf $\mathbb Q$-integrable distribution} of rank $r$ on $X$
to be the data of a line sheaf $L$  on $X$ 
and a non-zero morphism for some $m>0$ 
\[\phi\colon (\Omega^r_X)^{\otimes m} \rightarrow L\]
such that any generic point of $X$ admits a neighbourhood $U$ and
an integrable distribution $\mathcal F$ on $U$ defined by the Pfaff field $\phi_0\colon \Omega^r_U \rightarrow \mathcal O_U(K_{\mathcal F})$
and such that $\phi|_U = \phi_0^{\otimes m}$.
 We say that $m$ is the {\bf index}
of the $\mathbb Q$-integrable distribution.
We will refer to $\mathcal F$ as the {\bf associated integrable distribution}.

 \medskip
We make note of the following:

\begin{lemma}
\label{lem_sat_pfaff}
Let $X$ be an $S_2$ 
scheme, let $E$ be a coherent sheaf on $X$ and let 
$L_1, L_2$ be line sheaves on $X$ 
with morphisms $\psi_i\colon E \rightarrow L_i$ such that
\begin{enumerate}
\item $\psi_1 = \psi_2$ at the generic points of $X$; and

\item ${\rm coker}~ \psi_1$ is supported in codimension at least two.
\end{enumerate}

Then there exists a non-zero morphism $L_1 \rightarrow L_2$.
In particular, if $X$ is normal then there exists a
uniquely defined effective Weil divisor $B$ such that $L_2 = L_1\otimes \mathcal O_X(B)$.

\end{lemma}
\begin{proof}
We may freely remove subschemes of codimension at least two from $X$ and so we may assume that 
$L_1$ and $L_2$ are locally free and that $\psi_1$ is surjective.  Let $Q$ be the kernel of $\psi_1$.
By item (1) we have that $\psi_2(Q) \subset L_2$ is a torsion subsheaf, and is therefore identically
zero.  Our result then follows.
\end{proof}

\begin{lemma}
\label{lem_uniqueness}
Let $X$ be an $S_2$ scheme, let $\mathcal F^{\circ}$ be an integrable
distribution on $X$ and let $U \subset X$ be a dense open subset. 
\begin{enumerate}
\item If $X \setminus U$ is of codimension at least two  then $\mathcal F^{\circ}$ is 
uniquely determined by its restriction to $U$.

\item  If $\mathcal F^{\circ}$ is a foliation, then it is 
uniquely determined by its restriction
to $U$.  

\item Suppose that $X$ is normal and that $\mathcal G_U$ is a foliation on $U$. Then there exists
a unique foliation $\mathcal G$ on $X$ whose restriction to $U$ is $\mathcal G_U$.

\item Suppose that $X$ is normal.
Then there exists a unique foliation $\mathcal F$
on $X$ which agrees with $\mathcal F^\circ$ at the generic point of $X$. In particular, 
there exists a canonically defined 
effective divisor $B$ such that $K_{\mathcal F}+B \sim K_{\mathcal F^\circ}$.
\end{enumerate}

\end{lemma}
\begin{proof}
All the items are easy consequences of Lemma \ref{lem_sat_pfaff}.
\end{proof}

In Item (4), we will refer to $\cal F$ as 
the foliation induced by the integrable distribution
$\mathcal F^\circ$.

\subsection{Invariant subschemes}
Given an $S_2$ scheme $X$ and a rank $r$ integrable distribution $\mathcal F$ on $X$, we say that an irreducible
subscheme
$W \subset X$ is {\bf $\mathcal F$-invariant} (or simply {\bf invariant} if the integrable distribution is
understood) if $K_{\mathcal F}$ is Cartier at the generic point of $W$
and in a neighbourhood of the generic point of $W$
there is a factorisation

\begin{center}
\begin{tikzcd}
\Omega^r_X\vert_W \arrow[r, "\phi"] \arrow[d] & \mathcal O_W(K_{\mathcal F})\\
\Omega^r_W \arrow[ur] & 
\end{tikzcd}
\end{center}
More generally, we say that a subscheme $W \subset X$ is $\cal F$-invariant if each irreducible component
is invariant.

Given an $S_2$ scheme $X$, an integrable distribution $\mathcal F$  
and a  divisor $D$ on $X$, such that either $D$ is $\cal F$-invariant or  every component of 
$D$ is not $\cal F$-invariant, we define
$\epsilon(\mathcal F, D) \coloneqq 0$ if $D$ is $\mathcal F$-invariant, and $\epsilon(\mathcal F, D) \coloneqq 1$ if each component of $D$ is not $\cal F$-invariant.  
When $\mathcal F$ is clear from context we will write $\epsilon(D)$
in place of $\epsilon(\mathcal F, D)$.

Given any $\mathbb Q$-divisor $D$ on $X$, we denote  $D_{{\rm inv}}$ to be the part of $D$ supported
on invariant divisors and $D_{{\rm n-inv}} \coloneqq D-D_{{\rm inv}}$.

\begin{remark}
The definition above should be compared with \cite[Definition 3.4]{AD14}. Away from $\sing X \cup \sing \mathcal F$ these two definitions agree.
\end{remark}

\subsection{Foliations on normal varieties}
\label{s_normal}
When $X$ is a normal scheme, the above definition of foliation 
is equivalent to the usual definition of a 
foliation $\mathcal F$
in terms of a saturated subsheaf $T_{\mathcal F} \subset T_X$ closed under Lie bracket.
In this case, we may take $U\coloneqq X\setminus (\sing X\cup \sing\, \cal F)$, $E\coloneqq T_{\cal F}^*|_U$ and 
$L \coloneqq \mathcal O_X(K_{\mathcal F})$.
To verify this equivalence we first note that by Lemma \ref{lem_uniqueness} it suffices to verify this equivalence
away from a subvariety of  codimension at least two and, in particular, we may replace $X$ by $U$.
Given the saturated subsheaf $T_{\mathcal F} \subset T_X$ we get a Pfaff field
by considering the morphism
\[\Omega^r_X \rightarrow (\bigwedge^r T_{\mathcal F})^{*} 
\cong \mathcal O_X(K_{\mathcal F})\]
where $r$ is the rank of $\mathcal F$.

Conversely, consider a Pfaff field $\phi\colon \Omega^r_X \rightarrow L$ 
satisfying our integrability condition.  By assumption, over a dense open subset $U$ of $X$, we have 
a surjective morphism  $\Omega^1_U \rightarrow E$ whose dual defines a foliation on $U$, which extends uniquely to a foliation on $X$.

\begin{lemma}
\label{lem_invariant_vector_fields}
Let $X$ be a normal scheme,  let $\mathcal F$ be a rank $r$ foliation on $X$
and let $W \subset X$ be an irreducible subscheme 
such that $T_{\mathcal F}$ is locally free in a neighbourhood of the generic point of
$W$. Let $I_W$ be the ideal of $W$. 

If in a neighbourhood of the generic point of $W$ we have $\partial(I_W) \subset I_W$ for all local sections $\partial \in T_{\mathcal F}$
then $W$ is $\cal F$-invariant.
%
\end{lemma}
\begin{proof}
Everything is local about the generic point of $W$, so we may freely replace 
$X$ by a neighbourhood of the generic point of $W$ and therefore we may assume that $T_{\mathcal F}$
is locally free.
Let $\partial_1, \dots, \partial_r$ be generators of $T_{\mathcal F}$.

Suppose that  $\partial_i(I_W) \subset I_W$ for all $i$ then 
$(\partial_1 \wedge \dots\wedge \partial_r)(df \wedge \beta)$ vanishes along $W$ for any $f \in I_W$
and any $(r-1)$-form $\beta$.  In particular, 
$(dI_W\wedge \Omega_X^{r-1})\vert_W = \ker (\Omega^r_X\vert_W \rightarrow \Omega^r_W)$
is contained in the kernel of $\Omega^r_X\vert_W \rightarrow \mathcal O_W(K_{\mathcal F})$.
This implies that in a neighbourhood of the generic point of W,
the morphism factors through $\Omega^r_X\vert_W \rightarrow \Omega^r_W$ 
and so $W$ is $\cal F$-invariant.
%
%
%
\end{proof}

When $X$ is smooth the above Lemma is proven in \cite[Lemma 2.7]{AD13}.  Moreover it is shown there that 
if $W$ is $\cal F$-invariant and not contained in the singular
locus of $\mathcal F$ then in a neighbourhood of the generic point of $W$ we have $\partial(I_W) \subset I_W$ for all local sections 
$\partial \in T_{\mathcal F}$.

\subsection{Singularities of foliations from the perspective of the MMP}
We refer to \cite{KM98} for general notions of singularities 
coming from the MMP. 
We refer to \cite[\S 2.4]{CS18} for a recollection on the definition of foliation singularities
from the perspective of the MMP. 
We say that a variety $X$ is {\bf potentially klt} if there exists a $\mathbb Q$-divisor $\Gamma \geq 0$
such that $(X, \Gamma)$ is klt.

If $D$ is a $\mathbb Q$-divisor on a normal variety $X$ and $\Sigma$ is a prime divisor in $X$, 
then we denote by $m_\Sigma D$ the coefficient of $D$ along $\Sigma$.

\begin{lemma}
\label{rmk_1}
Let $X$ be a normal variety, let $\mathcal F$ be a foliation on $X$ and let $\Delta \geq 0$
be a $\mathbb Q$-divisor on $X$ such that $(\mathcal F, \Delta)$ is log canonical.  

Then
no component of $\Delta$ is $\mathcal F$-invariant.
\end{lemma}
\begin{proof} 
It follows immediately from \cite[Remark 2.3]{CS18}.
\end{proof}

\subsection{Algebraically integrable foliations}
A dominant map $\sigma \colon X\dashrightarrow Y$ between normal varieties is called {\bf almost holomorphic} if there exist dense Zariski open subsets $U\subset X$ and $V\subset Y$ such that the induced map $\sigma|_U\colon U\to V$ is a proper morphism. 

Let $\sigma\colon Y\dashrightarrow X$ be a dominant map between normal varieties and let $\mathcal F$ be a foliation of rank $r$ on $X$. We denote by $\sigma^{-1}\mathcal F$ the {\bf induced foliation} on $Y$ (e.g. see \cite[Section 3.2]{Druel21}).  
If $T_{\mathcal F}=0$, i.e. if $\mathcal F$ is the foliation by points on $X$, 
then we refer to $\sigma^{-1}\mathcal F$ as the  {\bf foliation induced by $\sigma$} and we denote it by $T_{X/Y}$. 
In this case, the foliation $\sigma^{-1}\mathcal F$ is called {\bf algebraically integrable}. 

Let $f\colon X\to Z$ be a morphism between normal varieties and let $\mathcal F$ be the induced foliation on $X$. 
If $f$ is equidimensional,
then we define the  {\bf ramification divisor} $R(f)$ of $f$  as 
\[
R(f)=\sum_D (f^*D-f^{-1}(D))
\]
where the sum runs through all the prime divisors of $Z$.  Note that, since $f$ is equidimensional,
the pullback $f^*D$ is well defined even though $D$ might not be $\mathbb Q$-Cartier. In this case, we have
\[K_{\mathcal F} \sim K_{X/Z}-R(f)\] 
(e.g. see  \cite[Notation 2.7 and \S 2.9]{Druel17}).

Let $X$ be a normal variety and let $\cal F$ be a foliation on $X$. Then 
$\mathcal F$ is called {\bf purely transcendental} if there is no positive dimensional algebraic subvariety passing 
through the general point of $X$, which is tangent to $\cal F$. 
In general,  by  \cite[Definition 2.3]{Druel17} (see also \cite[Definition 2]{AD17}) it follows that for any foliation $\cal F$ on $X$ there exists a dominant map 
$\sigma\colon X\dashrightarrow Y$ and a purely transcendental foliation $\cal G$ on $Y$ such that $\cal F=\sigma^{-1}\cal G$. 
 Note that $Y$ and $\cal G$ are unique up to birational equivalence. 
The foliation $\cal H$ induced by $\sigma$ is called the {\bf algebraic part of } $\cal F$.
\section{Adjunction}

\subsection{Lifting derivations on the normalisation}

The goal of this Subsection is to prove the following:

\begin{proposition}
\label{lem_tensor_lift}
Let $X$ be a reduced scheme 
and let $n\colon \tilde{X} \rightarrow X$ be its normalisation.
Suppose that $L$ is a locally free sheaf of rank one on $X$ and that we have a morphism
$\phi\colon (\Omega^r_X)^{\otimes m} \rightarrow L$ for some $r, m \geq 0$.  

Then there is a natural morphism
$\tilde{\phi} \colon (\Omega^r_{\tilde{X}})^{\otimes m} \rightarrow n^*L$ which agrees
with $\phi$ at any generic point of $\tilde{X}$.
\end{proposition}
Using the same notation as in the Proposition, we will call $\tilde{\phi}$ the {\bf lift } of $\phi$.

We follow closely the proofs of  \cite[Theorem 2.1.1]{Kal06} and \cite[Proposition 4.5]{MR2439607}. 
Given an integral domain $A$ we denote by $K(A)$ the field of fractions of $A$. Let $M$ be an $A$-module and let $r,m\ge 0$. Then an $A$-linear map 
$$\phi\colon (\Omega^r_{A})^{\otimes m} \to M$$
induces a map 
\[
\partial\colon K(A)^{\oplus rm}\to M\otimes_{A}K(A)\quad\text{such that}\quad \partial(A^{\oplus rm})\subset M.\]
Note that $\partial$ is a derivation if $r=m=1$. 

We  begin with the following two Lemma:

\begin{lemma}
\label{lem_g_seidenberg}
Let $A$  be a Noetherian integral $K$-algebra,  let 
$B \subset A$ be a Noetherian subalgebra and let 
$\partial\colon B \rightarrow A$ be a derivation.
Let $A'$ (resp. $B'$) be the integral closure of $A$ in $K(A)$ (resp. of $B$ in $K(B)$).

Then $\partial$ lifts to a derivation $\partial'\colon B' \rightarrow A'$.
\end{lemma}
\begin{proof}
The proof of \cite[Theorem, \S3]{MR0188247} works equally well here.
Indeed, as noted in  \cite[Footnote 2]{MR0188247}, the only thing that is needed in the proof is that the
differential operator $E\coloneqq e^{t\partial}$ defines an injective map 
$E\colon K(B)[[t]] \rightarrow K(A)[[t]]$, which is immediate since $B$ is a subalgebra of $A$.
\end{proof}

\begin{lemma}
\label{lem_root}
Let $X$ be a normal scheme and let $E$ be a coherent sheaf on $X$.
Let $m$ be a positive integer and let  $s\colon E^{\otimes m} \rightarrow \mathcal O_X$ be a morphism.
Assume that there exists a morphism $t\colon E\rightarrow \mathcal O_{X}$ and a rational 
function $\rho \in K(X)$ such that $s = \rho t^{\otimes m}$.  


Then there exists a cyclic Galois cover $\sigma\colon \overline{X} \rightarrow X$ and a 
morphism $\bar{t}\colon \sigma^*E \rightarrow \mathcal O_{\overline{X}}$ such that
$\bar{t}^{\otimes m} = \sigma^*s$.
\end{lemma}
\begin{proof}
Let $\overline{X}$ to be the normalisation 
of $X$ in $K(X)(\sqrt[m]{\rho})$ and let $\sigma\colon \overline{X}\to X$ be the induced morphism. Then there exists a rational function 
$z$ on $\overline{X}$ such that if $\bar  t=z \sigma^*t$ then 
$\bar{t}\colon \sigma^*E \rightarrow \mathcal O_{\overline{X}}$ is a morphism such that
$\bar{t}^{\otimes m} = \sigma^*s$.
\end{proof}

We call the morphism $\bar{t}$ constructed in Lemma \ref{lem_root} the {\bf $m$-th root} of $s$.

%
\begin{proof}[Proof of Propostion \ref{lem_tensor_lift}]
The claim is local on $X$ so we may freely assume that 
$X = \Spec ~A$ is affine and $L \cong \mathcal O_X$.
Moreover, it suffices to prove the existence of the lift after restricting to an irreducible component of $X$ and 
so we may freely assume that $X$ is integral.
Let $A'$ be the integral closure of $A$ in $K(A)$ 
and define $X' \coloneqq \Spec\, {A'}$ and let $n\colon X' \rightarrow X$ be the normalisation 
morphism.
Using the same argument as in the proofs of \cite[Lemma 4.3 and Proposition 4.5]{MR2439607}, we may assume that $A$ and $A'$ are complete one dimensional local rings.

Let us first suppose that $r= 1$. 
The map $\phi\colon (\Omega^1_X)^{\otimes m}\to \mathcal O_X$ induces a map $\phi'\colon (n^*\Omega^1_X)^{\otimes m}
\to \mathcal O_{X'}$.  Let $t\colon n^*\Omega^1_X
\to \mathcal O_{X'}$ be a non-zero map and let $\rho\in K(X')$ be a rational function such that $\phi'=\rho t^{\otimes m}$. 
Then Lemma \ref{lem_root} implies the existence of a cover 
$\sigma\colon \overline{X} = \Spec\, {\overline{A}} \rightarrow X'$  associated to $\phi'$.
Let 
$\psi\colon \sigma^*n^*\Omega^1_X \rightarrow \mathcal O_{\overline{X}}$ be the $m$-th root of $\phi$.

Observe that $\psi$ corresponds to a derivation 
$\partial_{\psi}\colon A \rightarrow \overline{A}$.
By Lemma \ref{lem_g_seidenberg} this lifts to a derivation
$\partial'_{\psi}\colon A' \rightarrow \overline{A}$.
This in turn implies that $\psi$ lifts to a morphism
$\rho\colon \sigma^*\Omega^1_{X'} \rightarrow \mathcal O_{\overline{X}}$.
Finally note that $\rho^{\otimes m}$ is $G$-invariant and so descends to
a morphism $(\Omega^1_{X'})^{\otimes m} \rightarrow \mathcal O_{X'}$, which is
precisely our required lifting of $\phi$.

\medskip

To prove the claim, by following the proof
of \cite[Proposition 4.5]{MR2439607}, we may proceed by induction on $r$ and
reduce the proof of the statement for general $r \geq 1$
to the case $r = 1$, proven above.
\end{proof}

\begin{lemma}
\label{lem_vanishing_lift}
Let $X$ be an integral scheme and let $\phi\colon \Omega_X^r \rightarrow L$ be a Pfaff field where $L$ is  a locally free sheaf of rank one on $X$.
Let $n\colon X' \rightarrow X$ be the normalisation.  Let $P \subset X$ be a codimension one subscheme and assume that $\phi$ vanishes along $P \subset X$.

Then $\phi'\colon \Omega^r_{X'} \rightarrow n^*L$ vanishes along $n^{-1}(P)$ where $\phi'$ is the lift
of $\phi$ to the normalisation.
\end{lemma}
\begin{proof}
The claim is local about the generic point of $P$, so up to shrinking $X$ we may freely assume that 
$X = \Spec\, A$ is affine and $X' = \Spec\, A'$.
We may argue as in the proof of \cite[Proposition 4.5]{MR2439607} to reduce to the case where $A$ and $A'$ are  complete one dimensional  local rings
with coefficients fields $K$ and $K'$, respectively.
We now proceed by induction on $r$.
When $r = 1$ the claim follows from \cite[Lemme 1.2]{MR2092774} (note that the cited Lemma is stated for varieties, but applies equally well to reduced schemes). 
So suppose that $r \geq 2$.  Let $t$ be a uniformising parameter of $A'$. 
To prove our claim, it suffices to show that $\phi'(dx_1\wedge\dots\wedge dx_r)=0$ for $x_1,\dots,x_r\in K'\cup \{t\}$. 
Since $dt\wedge \dots \wedge dt = 0$ we may assume (up to relabeling)
that $x_1 \in K'$.  Since $K \subset K'$ is a finite extension we may find $Q(s) = \sum_{i = 0}^c a_is^i \in K[s]$ such that $Q(x_1) =0$ and $Q'(x_1) \neq 0$.
On one hand we have that  
\[0 =dQ(x_1)\wedge dx_2 \wedge \dots \wedge dx_r = Q'(x_1)dx_1\wedge \dots \wedge dx_r + \sum_{i = 0}^c x_1^ida_i \wedge dx_2 \wedge \dots \wedge dx_r\]
and so 
\[
\phi'(dx_1\wedge \dots \wedge dx_r) = -\frac{1}{Q'(x_1)}\phi'(\sum_{i = 0}^c x_1^ida_i \wedge dx_2 \wedge \dots \wedge dx_r).
\]

On the other hand, $\phi(da_i \wedge \cdot)\colon \Omega^{r-1}_X \rightarrow L$ defines a Pfaff field of rank $r-1$
and so our induction hypothesis implies that 
\[
\phi'(\sum_{i = 0}^c x_1^ida_i \wedge dx_1 \wedge \dots \wedge dx_r)
\]
vanishes along $n^{-1}(P)$ and we may conclude.
\end{proof}

\subsection{Construction of the different}


\begin{lemma}
\label{lem_cartier_adj}
Let $X$ be a normal scheme, let $\mathcal F$ be a foliation of rank $r$ on $X$,
let $\iota \colon D \hookrightarrow X$ be a reduced 
subscheme of codimension one and suppose that either every component of $D$ is $\cal F$-invariant or that every component of 
$D$ is not $\cal F$-invariant.
Let $n\colon S \rightarrow D$ be the normalisation and suppose there exist
\begin{enumerate}
\item a  subscheme $Z$ of $X$ such that $Z \cap D$ is of codimension at least two in $D$; and 
\item a $\mathbb Q$-divisor $\Delta \geq 0$ on $X$
which does not  contain any component of $D$ in its support and such that for some sufficiently divisible positive integer $m$ we have that
$m(K_{\mathcal F}+\Delta)$ and $\epsilon(D)D$ are Cartier 
on $X \setminus Z$.
\end{enumerate}

Then, there exists a canonically defined $\mathbb Q$-integrable distribution  of rank $r-\epsilon(D)$ and index $m$ on $S$, given by a morphism
\[\psi_S\colon  (\Omega^{r-\epsilon(D)}_S)^{\otimes m} \rightarrow 
n^w\mathcal O_X(m(K_{\mathcal F}+\Delta+\epsilon(D)D)).\]
\end{lemma}
\begin{proof}
It suffices to prove the Lemma away from a subvariety of codimension 
at least two in $D$, and so we may freely
assume that for some sufficiently divisible positive integer $m$, we have that  
$m(K_{\mathcal F}+\Delta)$ and $\epsilon(D)D$ are Cartier.

Taking the $m$-th tensor power of the Pfaff field defining $\mathcal F$ 
and composing 
with the inclusion $\mathcal O_X(mK_{\mathcal F}) 
\rightarrow \mathcal O_X(m(K_{\mathcal F}+\Delta))$
we get a morphism
$\phi\colon (\Omega^r_X)^{\otimes m} \rightarrow \mathcal O_X(m(K_{\mathcal F}+\Delta))$.

Suppose first that $D$ is $\cal F$-invariant.
Let $N \coloneqq {\rm ker}~ ((\Omega^r_X)^{\otimes m}\vert_D \rightarrow (\Omega^r_D)^{\otimes m})$
and let $\phi\vert_D$ be the restriction of  $\phi$ to $D$. 
By definition $\phi\vert_D(N)$ vanishes at the generic point of $D$, and since $m(K_{\mathcal F}+\Delta)$
is Cartier and $D$ is reduced, it follows that $\phi\vert_D(N) = 0$.
Therefore we have a morphism 
$\psi\colon (\Omega^r_D)^{\otimes m} \rightarrow \mathcal O_D(m(K_{\mathcal F}+\Delta))$.

By Lemma  \ref{lem_invariant_vector_fields}, we have a commutative diagram in a neighbourhood $U$ of
the generic point of $D$
\begin{center}
\begin{tikzcd}
\Omega^1_U|_{U\cap D}  \arrow[r]\arrow[d] &T^*_{\cal F}\vert_{U\cap D}\\
\Omega^1_{U\cap D} \arrow[ur] & 
\end{tikzcd}
\end{center}
and the integrability condition follows immediately. 

\medskip

Now suppose that $D$ is not $\cal F$-invariant.  We define
a morphism 
\[
\psi'\colon (\Omega_D^{r-1})^{\otimes m} \rightarrow \mathcal O_D(m(K_{\mathcal F}+\Delta+D))
\]
 by 
\[
\psi'(\alpha_1\otimes \dots \otimes \alpha_m) \coloneqq
\frac {\phi(df \wedge \tilde{\alpha}_1 \otimes \dots \otimes df \wedge \tilde{\alpha}_m)}{f^m}\vert_D
\]
 for any  local sections $\alpha_1, \dots, \alpha_m$ of $\Omega^{r-1}_D$, where
 $f$ is a local equation of $D$,
$\tilde{\alpha}_i$ is any local lift of $\alpha_i$ to $\Omega^{r-1}_X$,
and  $\frac {\phi(df \wedge \tilde{\alpha}_1 \otimes \dots \otimes df \wedge \tilde{\alpha}_m)}{f^m}$ 
is considered as 
a section of $\mathcal O_X(m(K_{\mathcal F}+\Delta+D))$.
We claim that this morphism is independent
of our choice of $f$. Indeed, if $f'$ is another local equation of $D$
then $f' = uf$ where $u$ is a unit.  We compute
that 
\[
\frac{\phi(df'\wedge \tilde{\alpha}_1 \otimes \dots \otimes df'\wedge \tilde{\alpha}_m)}{(f')^m} = 
\frac{\phi(df\wedge \tilde{\alpha}_1 \otimes \dots \otimes df \wedge \tilde{\alpha}_m)}{f^m}+
\frac{\phi(\omega_0)}{f^{m-1}},
\] 
where $\omega_0$ is a local section of $\Omega^r_X$. 
Observe that 
$\frac{\phi(\omega_0)}{f^{m-1}}$ 
 is a section of 
$\mathcal O_X(m(K_{\mathcal F}+\Delta)+(m-1)D)$, 
and so it vanishes along $D$ when considered as a section of $\mathcal O_X(m(K_{\mathcal F}+\Delta+D))$.  Thus, 
\[
\frac{\phi(df'\wedge \tilde{\alpha}_1 \otimes \dots \otimes df'\wedge \tilde{\alpha}_m)}{(f')^m}|_D = 
\frac{\phi(df\wedge \tilde{\alpha}_1 \otimes \dots \otimes df \wedge \tilde{\alpha}_m)}{f^m}|_D,
\] 
as required.
Likewise, it is easy to check that $\psi'$ is independent of the choice of $\tilde{\alpha_i}$. 

Since $D$ is not $\cal F$-invariant, by Lemma  \ref{lem_invariant_vector_fields} the composition 
\[
\cal O_{U\cap D}(-D)\to \Omega^1_U|_{U\cap D}\to T_{\cal F}^*|_{U\cap D}
\]
is non-zero. Let $E$ be its cokernel. Then, the induced map 
\[
\Omega^1_{U\cap D}\to E
\]
satisfies the integrability condition.

In either case, by Proposition \ref{lem_tensor_lift} we have a lift of $\psi$ or $\psi'$
to 
\[\tilde{\psi}\colon (\Omega^{r-\epsilon(D)}_S)^{\otimes m} 
\rightarrow n^*\mathcal O_X(m(K_{\mathcal F}+\Delta+\epsilon(D)D))
\]
which necessarily satisfies our integrability condition and we may conclude.
\end{proof}

\begin{remark}
\label{r_adjunction1}
 Set-up as in Lemma \ref{lem_cartier_adj}. If in addition we have that 
$D$ is   $S_2$, then using the same proof as in the Lemma,   it follows that 
there exists a canonically defined $\mathbb Q$-integrable distribution  of rank $r-\epsilon(D)$ and index $m$ on $D$, given by a morphism
\[\psi_D\colon (\Omega_D^{r-\epsilon(D)})^{\otimes m} \rightarrow \iota^w\mathcal 
O_X(m(K_{\mathcal F}+\Delta+\epsilon(D)D)).\]
\end{remark}

\begin{propdef}
\label{prop_adjunction}
Let $X$ be a normal scheme, let $\mathcal F$ be a foliation of rank $r$ 
on $X$, and let $\iota \colon D \hookrightarrow X$ 
be an integral subscheme of codimension one.
Let $n\colon S \rightarrow D$ be the normalisation.
Suppose that there exist a  subscheme $Z$ of $X$ such that $Z \cap D$ is 
of codimension at least two in $D$ and 
a $\mathbb Q$-divisor $\Delta\ge 0$ on $X$
which does not  contain $D$ in its support and
such that $K_{\mathcal F}+\Delta$ and $\epsilon(D)D$ are $\mathbb Q$-Cartier on $X \setminus Z$.

Then, there exists a canonically defined {\bf restricted foliation}
$\mathcal F_{S}$ on $S$, and a canonically defined $\mathbb Q$-divisor ${\rm Diff}_S(\mathcal F, \Delta) \geq 0$, 
called the {\bf different}, 
such that 
\[
n^w(K_{\mathcal F}+\Delta+\epsilon(D)D) \sim_{\mathbb Q} K_{\mathcal F_S}+{\rm Diff}_S(\mathcal F, \Delta). \]
If $\Delta =0$ then we denote ${\rm Diff}_S(\cal F,0)$ simply as ${\rm Diff}_S(\cal F)$. 
\end{propdef}

\begin{proof}
It suffices prove the Proposition away from a subvariety of codimension 
at least two in $D$, and so we may freely
assume that $K_{\mathcal F}+\Delta$ and $\epsilon(D)D$ are $\mathbb Q$-Cartier.

We first construct the restricted foliation on $S$.
By Lemma \ref{lem_uniqueness}
it suffices to define the restricted foliation $\mathcal F_S$ at the generic point of $S$. 
We may therefore assume that $K_{\mathcal F}$
and $D$ are both Cartier and therefore may apply Lemma \ref{lem_cartier_adj} to produce $\mathcal F_S$.

Let $m$ be a sufficiently divisible positive integer $m$ such that 
$m(K_{\mathcal F}+\Delta)$ and $m\epsilon(D)D$ are Cartier. 
Then, by Lemma \ref{lem_cartier_adj}, we have a morphism at the generic point $\eta$ of $S$, 
\[\psi_\eta\colon  (\Omega^{r-\epsilon(D)}_S)^{\otimes m}\vert_\eta \rightarrow 
\mathcal O_S(m(K_{\mathcal F}+\Delta+\epsilon(D)D))\vert_\eta\]
which agrees with $(\Omega^{r-\epsilon(D)}_S)|_{\eta}^{\otimes m} \rightarrow \mathcal O_\eta(mK_{\mathcal F_S})$.  
If we can show that, after possibly replacing $m$ by a larger multiple,  $\psi_{\eta}$ extends to a morphism on all of $S$, call it $\psi_S$, 
then by Lemma \ref{lem_sat_pfaff} there exists an effective Weil divisor $B$ such that 
\[
mK_{\cal F_S}+B\sim m(K_{\cal F}+\Delta+\epsilon (D)D)|_S.
\]
Thus, it is enough to define ${\rm Diff}_S(\cal F,\Delta)\coloneqq \frac 1 m B$.

\medskip

%
%
%
%
%


Note that if $D$ is $\cal F$-invariant, then the existence of $\psi_S$ is guaranteed by Lemma  \ref{lem_cartier_adj}. 
Thus, we may assume that $D$ is not $\cal F$-invariant. To check that $\psi_{\eta}$ extends as a morphism, 
it suffices to do so locally.  Let $P\in D$ be a closed point. Then there exists a quasi-\'etale cyclic cover 
$q\colon V\to U$ with Galois group $G$, where $U$ is a neighbourhood of $P$ in $X$, and 
such that 
$q^*D$ is a Cartier divisor. Let $m$ be the order of $G$. After possibly replacing $q$ by a higher cover, we may assume that $m$ does not depend  on $P\in D$. Let
 $\cal F_V\coloneqq q^{-1}\cal F|_U$ be the foliation induced on $V$, let $D_V\coloneqq q^{-1}(D \cap U)$ and let $\Delta_V$ be a $\mathbb Q$-divisor on $V$ such that $K_{\cal F_V}+\Delta_V=q^*(K_{\cal F}+\Delta)$.
By assumption, we may assume that $q$ is \'etale in a neighborhood of any 
general point of  $U\cap {\rm Supp}\, \Delta$.
Thus, $\Delta_V\ge 0$. Let $\nu\colon S_V\to D_V$ be the normalisation of $D_V$ 
and let  $p\colon S_V\to S_U:=n^{-1}(U\cap D)$ be the induced morphism. 
Note that $D_V$ is not $\cal F_V$-invariant.

Since $m(K_{\cal F_V}+\Delta_V)$ and $D_V$ are Cartier on $V$, 
by Lemma \ref{lem_cartier_adj} there exists a $\mathbb Q$-integrable distribution of rank $r-1$ and index $m$ on $S_V$, 
given by a morphism
\[
\psi_{S_V}\colon (\Omega_{S_V}^{r-1})^{\otimes m}\to \nu^*\mathcal O_{D_V}(m(q^*(K_{\cal F}+\Delta+D))).
\]
Since $S_U$ is $S_2$, and $m(K_{\cal F}+\Delta+D)$ is Cartier on $S_U$
the map $\psi_\eta|_{S_U}$ extends to $S_U$ if and only if the map $p^*(\psi_\eta|_{S_U})$
extends to $S_V$.
If $\sigma\colon p^*(\Omega^{r-1}_{S_U})^{\otimes m}\to  (\Omega_{S_V}^{r-1})^{\otimes m}$ denotes the natural map then 
the composition 
\[
\psi_{S_V}\circ \sigma \colon p^*(\Omega^{r-1}_{S_U})^{\otimes m}\to \nu^*\mathcal O_{D_V}(m(q^*(K_{\cal F}+\Delta+D)))
\]
is an extension of $p^*(\psi_\eta|_{S_U})$.
Thus, our claim follows.
\end{proof}

\begin{remark}
Set-up as in Proposition-Definition \ref{prop_adjunction}. If in addition we have that  $D$ is  $S_2$ then there exists a  
 $\mathbb Q$-integrable distribution of rank $r-\epsilon(D)$ and index $m$ on $D$, given by the 
morphism    
\[\psi_D\colon  (\Omega^{r-\epsilon(D)}_D)^{\otimes m} \rightarrow 
\iota^w\mathcal O_X(m(K_{\mathcal F}+\Delta+\epsilon(D)D))\]
and whose associated integrable distribution on $D$ coincides with the restricted foliation on $S$ at any generic point of $D$. 
Indeed, using the same construction 
as in Proposition-Definition \ref{prop_adjunction},
by Lemma \ref{lem_cartier_adj} and Remark \ref{r_adjunction1} we may assume that $D$ is not $\cal F$-invariant and that 
we have a morphism
\[
(\Omega^{r-\epsilon(D)}_{D_V})^{\otimes m} \rightarrow  \mathcal O_{D_V}(q^*(m(K_{\cal F}+\Delta+D))).
\]
Proceeding as in the proof of Proposition-Definition \ref{prop_adjunction}, 
it follows that there exists a morphism
\[
(\Omega^{r-1}_{D_U})^{\otimes m} \rightarrow
\iota^w(\mathcal O_U(m(K_{\mathcal F}+\Delta+D)))
\]
and so we may conclude.
%
\end{remark}

\begin{remark}
Set-up as in Proposition-Definition \ref{prop_adjunction}.
If in addition we have that  $K_{\mathcal F}+\Delta$ and $\epsilon(D)D$ are $\mathbb Q$-Cartier then 
\[n^*(K_{\mathcal F}+\Delta+\epsilon(D)D) \sim_{\mathbb Q} 
K_{\mathcal F_S}+{\rm Diff}_S(\mathcal F, \Delta).\]
\end{remark}

\begin{remark}\label{r_diff_b}
Set up as in Definition-Proposition \ref{prop_adjunction}. Let $B\ge 0$ be a  $\mathbb Q$-divisor which does not contain  $D$ in its support and such that 
 $B$ is $\mathbb Q$-Cartier
on $X \setminus Z$.
Then, from the construction it follows immediately that 
\[
{\rm Diff}_S(\cal F,\Delta+B) = {\rm Diff}_S(\cal F,\Delta)+n^wB.
\]
\end{remark}

\begin{remark}
\label{rmk_2}
Set up as in Definition-Proposition \ref{prop_adjunction}. We can compute the different using resolutions.  Indeed,
suppose that  $\pi\colon X' \rightarrow X$ is a log resolution of $(X, D+{\rm Supp}\, \Delta)$. Note that $\pi$ is not necessarily
a reduction of singularities of $\mathcal F$. We may write 
\[
K_{\mathcal F'}+\Delta'+ \epsilon(D)D'+E = \pi^*(K_{\mathcal F}+\Delta+\epsilon(D)D)
\]
where $E$ is $\pi$-exceptional, $\mathcal F' \coloneqq \pi^{-1}\mathcal F$, $D'= \pi_*^{-1}D$
and $\Delta' = \pi_*^{-1}\Delta$.
Let $S$ be the normalisation of $D$ and let $\mu\colon D' \rightarrow S$ be the induced morphism.

Then  
\[
\mu_*({\rm Diff}_{D'}(\cal F',\Delta')+E|_{D'}) = {\rm Diff}_S(\cal F,\Delta).
\]
\end{remark}

\medskip

In the case of higher codimension invariant centres, we also have a sub-adjunction statement:

\begin{propdef}
\label{prop_inv_subadj_analytic}
Let $X$ be a normal complex analytic space, let $\mathcal F$ be a foliation of rank $r$ 
on $X$, and let $\iota \colon D \rightarrow X$ 
be a closed analytic subset. 
Let $n\colon S \rightarrow D$ be the normalisation.

Suppose that 
\begin{enumerate}
\item 
there exists a $\mathbb Q$-divisor $\Delta\ge 0$ such that $K_{\mathcal F}+\Delta$ is $\mathbb Q$-Cartier;


\item $D$ is not contained in ${\rm Supp }\,\Delta\cup \sing\, {\mathcal F}$; 

\item $T_{\mathcal F}$ is locally free in a neighbourhood of the general point of $D$; and 
\item $D$ is $\mathcal F$-invariant.
\end{enumerate}

Then there exists a canonically defined {\bf restricted foliation}
$\mathcal F_{S}$ on $S$, and a canonically defined $\mathbb Q$-divisor ${\rm Diff}_S(\cal F,\Delta) \geq 0$, 
called the {\bf different}, 
such that 
\[
n^*(K_{\mathcal F}+\Delta) 
\sim_{\mathbb Q} K_{\mathcal F_S}+{\rm Diff}_S(\cal F,\Delta).
\] 
\end{propdef}
\begin{proof}
We first construct the restricted foliation $\mathcal F_S$.  
Set 
\[
N \coloneqq \ker (\Omega^1_X \rightarrow T_{\mathcal F}^*)
\] and let $N_S$ be the image of $n^*N$ under the natural map 
$n^*\Omega^1_X \rightarrow \Omega^1_S$.  We then define $\mathcal F_S$ by taking $T_{\mathcal F_S}$ to be the annihilator of 
$N_S$.  Since $T_{\mathcal F}$ is locally free in a neighbourhood of the general point of $D$ it follows that 
the rank of $T_{\mathcal F_S}$ is equal to the rank of $T_{\mathcal F}$.


The remaining claims then follow  as in the proof of
 Proposition-Definition \ref{prop_adjunction}.  We briefly sketch this now.
 Let $m>0$ be such that $m(K_{\mathcal F}+\Delta)$ is Cartier and set $L \coloneqq \mathcal O_X(m(K_{\mathcal F}+\Delta))$.  
 Since $D$ is $\cal F$-invariant, the restriction of the morphism 
 $\phi \colon (\Omega^r_X)^{\otimes m} \rightarrow L$ to $D$ factors through 
 $\psi \colon (\Omega^r_D)^{\otimes m} \rightarrow L\vert_D$.  This lifts to a morphism 
 $\tilde{\psi} \colon (\Omega^r_S)^{\otimes m} \rightarrow n^*L$ by Proposition \ref{lem_tensor_lift} (we remark that constructing this lift is local 
 about any point of $x \in D$, and so by considering $\psi$ as a multi-derivation on $\mathcal O_{D, x}$, we see that the proof of Proposition 
\ref{lem_tensor_lift} works equally well here).

We conclude by observing that around a general point of $S$, $\tilde{\psi}$ is the $m$-th tensor power of the Pfaff field 
$\Omega^r_S \rightarrow \mathcal O_S(K_{\mathcal F_S})$ and so we deduce that there exists an effective divisor $B$ such that 
$L|_S \cong \mathcal O_S(mK_{\mathcal F_S}+B)$.  We conclude by taking ${\rm Diff}_S(\cal F,\Delta) \coloneqq \frac{1}{m}B$.
\end{proof}

\begin{remark}
Set up as in Definition-Proposition \ref{prop_inv_subadj_analytic}. If $r=1$ then (3) is automatically satisfied since, by assumption,  $K_{\cal F}$ is Cartier at the generic point of $D$. 

Note moreover that, under the set-up of Proposition-Definition \ref{prop_adjunction}, if $D$ is an $\cal F$-invariant
integral subscheme of codimension one, then the construction of $(\mathcal F_S, {\rm Diff}_S(\cal F,\Delta))$ in Proposition-Definition \ref{prop_adjunction} coincides with its construction in Proposition-Definition \ref{prop_inv_subadj_analytic}.
\end{remark}

\subsection{Calculation of the different in some special cases}

The following proposition computes the different in some special cases:

\begin{proposition}
\label{prop_explicit_calc}
Let $X$ be a normal scheme, let $\mathcal F$ be a foliation of rank $r$ 
on $X$, let $D \subset X$ be a reduced subscheme of codimension one such that  either every component of $D$ is $\cal F$-invariant or  every component of 
$D$ is not $\cal F$-invariant.
Suppose that $K_{\mathcal F}$ and $\epsilon(D)D$ are $\mathbb Q$-Cartier.
Let $n\colon S \rightarrow D$ be the normalisation, let $\cal F_S$ be the restricted foliation on $S$ 
 and 
let $P$ be a prime divisor on $S$.

Then the following hold:

\begin{enumerate}
\item If $n(P)$ is contained in $\sing\, {\mathcal F}$ then 
$m_P{\rm Diff}_S(\cal F) >0$.

\item Suppose that $D$ is Cartier, $\epsilon(D) = 1$ and that  
$T_{\mathcal F}$ is locally free and generated by $\partial_1, \dots, \partial_r$.
If $f$ is a local parameter for $D$ and 
$(\partial_1(f), \dots, \partial_r(f)) \subset I_{n(P)}$ then $m_P{\rm Diff}_S(\cal F)>0$.
Moreover, $m_P{\rm Diff}_S(\cal F)$ is an integer.
Conversely, suppose that $X$ and $D$ are smooth at the generic point of $n(P)$,
and that $m_P{\rm Diff}_S(\cal F)>0$ then 
$(\partial_1(f), \dots, \partial_r(f)) \subset I_{n(P)}$.

\item Suppose that $\epsilon(D) = 0$ and that $n(P)$ is not contained
in $\sing\, {\mathcal F}$.  
Then $m_P{\rm Diff}_S(\cal F) = \frac{m-1}{m}\epsilon(\mathcal F_S, P)$ where $m$ is a positive integer that divides the
 Cartier index of
$K_{\mathcal F}$ at the generic point of $n(P)$.  
If the index one cover associated to $K_{\mathcal F}$
is smooth in a neighbourhood of the preimage of the generic point of $n(P)$,
then $m$ equals the Cartier index of $K_{\mathcal F}$ 
at the generic point of $n(P)$.

\item Suppose that $\epsilon(D) = 1$ and that 
$X$ has quotient singularities in a neighbourhood of the
generic point of $n(P)$.  
Then $m_P{\rm Diff}_S(\cal F)$
is of the form $\frac{a+\epsilon(\mathcal F_S, P)(m-1)}{m}$ where $a \geq 0$ is an integer and 
where $m$ divides the Cartier index of $K_{\mathcal F}$ at the generic point of $n(P)$.
\end{enumerate}

\end{proposition}
\begin{proof}
All these assertions are local about the generic point of $n(P)$, so at any point 
we may freely replace $X$
by a neighbourhood of the generic point of $n(P)$.

\medskip

Proof of item (1).  Let us first suppose that $K_{\mathcal F}$ and $\epsilon(D)D$ are 
Cartier in a neighbourhood of $P$.
By assumption, the morphism 
$\phi\colon \Omega^r_X \rightarrow \mathcal O_X(K_{\mathcal F})$ takes values 
in $I\mathcal O_X(K_{\mathcal F})$ where $I$ is an ideal sheaf, 
whose co-support contains $n(P)$.

If  $\epsilon(D) = 0$ then  by Lemma \ref{lem_vanishing_lift} the lift of the restriction of $\phi$ to $D$ 
factors through 
\[
\Omega^r_S \rightarrow \mathcal O_S(n^*K_{\mathcal F} - P).
\]
If $\epsilon(D) = 1$ then, as in the proof of Lemma \ref{lem_cartier_adj}, 
we may define  $ \psi'\colon \Omega^{r-1}_X \rightarrow \mathcal O_X(K_{\mathcal F}+D)$ by 
$ \psi'\coloneqq \frac {\phi(df \wedge \cdot)}f$ for some choice $f$ 
of a local parameter for $D$. 
Then $\psi'$ still factors
through $I\mathcal O_X(K_{\mathcal F}+D)$ and so by Lemma \ref{lem_vanishing_lift} the lift of the 
restriction to $D$ factors through 
\[
\Omega^{r-1}_S \rightarrow \mathcal O_S(n^*(K_{\mathcal F}+D)- P).
\] 
In both cases, we have that 
$m_P{\rm Diff}_S(\cal F)\ge 1$, 
as required.

We now consider the general case.  Let  $q\colon V\to U$ be  a quasi-\'etale cyclic cover,  
where $U$ is a neighbourhood of the generic point of $P$ in $X$, and 
such that $q^*K_{\cal F}$ and $q^*(\epsilon(D)D)$ are Cartier divisors.  Let $\cal F'=q^{-1}\cal F$. 
%
%
%
%
%
 \cite[Corollary 5.14]{Druel21} guarantees that $P$ is contained
in $\sing\, {\mathcal F}$ if and only 
$P'\coloneqq q^{-1}(P)$ is contained
in $\sing\, \mathcal F'$. 
Let $S'$ denote the normalisation of $q^{-1}(D)$ and let $\cal F_{S'}$ and ${\rm Diff}_{S'}(\cal F')$
be respectively the restricted foliation and the different associated to $\mathcal F'$ on $S'$.
We have a finite morphism $p\colon S' \rightarrow S$
and we have that 
\[
K_{\mathcal F_{S'}}+{\rm Diff}_{S'}(\cal F') = p^*(K_{\mathcal F_S}+{\rm Diff}_S(\cal F)).
\]

Let $e$ be the ramification index of $p$ along $P'$.
By applying 
\cite[Lemma 3.4]{Druel21} 
(see also \cite[Proposition 2.2]{CS20}) we get an equality
\begin{align}
\label{rw_comp}
m_{P'}{\rm Diff}_{S'}(\cal F') = em_P{\rm Diff}(\cal F) - \epsilon(\cal F_S,P)(e-1).
\end{align}
Since $m_{P'}{\rm Diff}_{S'}(\cal F')\geq 1$ this implies that $m_P{\rm Diff}_S(\cal F) >0$.

\medskip

Proof of item (2). In this case, the Pfaff field 
 $\phi\colon \Omega^r_X \rightarrow \mathcal O_X(K_{\cal F})$
associated to $\mathcal F$ 
is given explicitly  by 
$\omega\mapsto \omega(\partial_1\wedge \dots \wedge \partial_r)$ where $\omega$
is a local section of $\Omega^r_X$.  
It follows that if $\eta$ is any local section of $\Omega^{r-1}_X$  then 
\[
\phi(df \wedge \eta) = \sum_{i=1}^r (-1)^i \partial_i(f) 
\eta(\partial_1 \wedge \dots \wedge \hat{\partial_i} \wedge \dots \wedge \partial_r).
\]
So if $(\partial_1(f), \dots, \partial_r(f)) \subset I_{n(P)}$, then Lemma \ref{lem_vanishing_lift} implies that
the restricted Pfaff field 
$\Omega^{r-1}_S \rightarrow n^*(\mathcal O_X(K_{\mathcal F}+D))$  
vanishes along $P$.
Since $K_{\cal F}$ is Cartier, it follows that  $m_P{\rm Diff}_S(\cal F)$ is an integer.

Conversely, suppose that $D$ and $X$ are smooth at the generic
point of $P$ and there exists an $i$ such that $\partial_i(f) \notin I_P$.
Up to
relabelling we may take $i = 1$ and after 
localising about the generic point of $P$ we may assume $\partial_1(f)$ is a unit.  For $j\geq 2$, we define
\[
\partial_j' \coloneqq \partial_j - \frac{\partial_j(f)}{\partial_1(f)}\partial_1.
\]
We have that $\partial_1, \partial_2', \dots, \partial_r'$ generate $T_{\mathcal F}$
in a neighbourhood of the generic point of $P$ and so, up to replacing $\partial_j$ by $\partial'_j$,
we may freely assume that $\partial_j(f) = 0$ for $j \geq 2$.  A direct calculation as above
shows that the restricted Pfaff field 
$ \Omega^{r-1}_D \rightarrow \mathcal O_D(K_{\mathcal F}+D)$
does not vanish along $P$, as required.

\medskip

Proof of item (3). Suppose for the moment that $K_{\mathcal F}$ is Cartier. 
In this case, the natural map $\Omega^r_D \rightarrow \mathcal O_D(K_{\mathcal F})$
vanishes along $P$ if and only if 
$\Omega^r_X \rightarrow \mathcal O_X(K_{\mathcal F})$ vanishes along
$P$, and so $m_P{\rm Diff}_S(\cal F) = 0$ as required.

We now handle the case where $K_{\mathcal F}$ is $\mathbb Q$-Cartier. After possibly shrinking $X$, 
we consider the index one cover  $q\colon X' \rightarrow X$ associated to 
$K_{\mathcal F}$ and let $S'$ be the normalisation of $q^{-1}(D)$.
Note that the ramification index of $S' \rightarrow S$ along $P' \coloneqq q^{-1}(P)$
divides the Cartier index of $K_{\mathcal F}$ and that if $\cal F'\coloneqq q^{-1}\cal F$ 
then \cite[Corollary 5.14]{Druel21} implies that $P'$ is not contained in $\sing\, \cal F'$. 
We may then argue as in the proof of Item (1) and use Equation \eqref{rw_comp} to conclude.

To check the last claim, it suffices to consider the case where $\epsilon(\mathcal F_S, P) = 1$.
After possibly replacing $X$ by a neighbourhood of the generic point of $P$, we may assume that 
the index one cover $q\colon X' \rightarrow X$ associated to $K_{\cal F}$ is such that $X'$ is smooth and it is totally ramified along $P$. 
In particular, we may assume that  $D'\coloneqq q^{-1}(D)$
is connected. 
Let $m$ be the Cartier index of $K_{\mathcal F}$ and let  $\cal F'\coloneqq q^{-1}\cal F$.
Since $T_{\mathcal F'}$ is reflexive, by \cite[Corollary 1.4]{Hartshorne80} 
it is locally free away 
from a subset of codimension at least three in $X$. Thus, we may assume that 
$T_{\mathcal F'}$ 
is locally free. Let $P'\coloneqq q^{-1}(P)$. 

We claim that $D'$
is normal and irreducible.  Assuming the claim we see that the 
ramification index of the induced morphism $D' \rightarrow D$ 
along $P'$ is $m$, and we may conclude by Equation  \eqref{rw_comp}.

To prove the claim, first recall that $D'$ is connected.
Suppose that $D'$ is not smooth in a neighbourhood of $P'$.   
By 
\cite[Theorem 5]{MR0212027} for any local generator $\partial \in T_{\mathcal F'}$ 
we have that $\partial(I_{\sing {D'}}) \subset  I_{\sing {D'}}$.
Thus if we denote by $\mathcal F_{S'}^\circ$ the integrable distribution induced on $S'$, where $n'\colon S'\to D'$ is the normalisation
and  whose existence is guaranteed by Lemma
\ref{lem_cartier_adj},
 then for any local generator $\partial \in T_{\mathcal F_{S'}^\circ}$
we have $\partial(I_{P'}) \subset I_{P'}$.
 Since we assumed that $P$ is 
not $\mathcal F_S$-invariant, it follows that $P'$ is not $\mathcal F_{S'}$-invariant and so 
there is a local generator $\delta \in T_{\mathcal F_{S'}}$ such that in a neighbourhood of the generic point of 
$P'$,  $\delta(I_{P'})$ is not contained in $I_{P'}$.  On the other hand, $P'$ is ${\mathcal F_{S'}^\circ}$-invariant and so
it follows that in a neighbourhood of the generic point of $P'$ there is some local section $f \in I_{P'}$ 
such that $f\delta$ is a local generator of $T_{\mathcal F_{S'}^\circ}$.  Thus, some 
local generator of $T_{\mathcal F_{S'}^\circ}$ must
vanish along $P'$, which in turn implies that some local generator
of $T_{\mathcal F'}$ must vanish along $n'(P')$ and, in particular, $n'(P')$ is contained in $\sing\, \cal F'$.  
\cite[Corollary 5.14]{Druel21}
implies that $n(P)$ is contained in $\sing\, {\mathcal F}$, contrary to our hypothesis.

\medskip

Proof of item (4). 
As in the proof of item (1) using Equation \ref{rw_comp} we may freely replace $X$ by 
a quasi-\'etale cover and 
so may assume that $X$ is smooth. 
By \cite[Corollary 1.4]{Hartshorne80} we may assume that $T_{\mathcal F}$ is locally free.
We may then conclude by Item (2).
\end{proof}

%
%

\subsection{Adjunction of singularities}

\begin{lemma}
\label{lem_weak_resolution}
Let $X$ be a smooth variety, let $\mathcal F$ be a foliation of rank $r$ on $X$ and 
let $D$ be a smooth divisor on $X$
such that $D$ is not $\mathcal F$-invariant.  Let $Z \subset D$ be a codimension one subvariety in $D$.

Then there exists a log resolution $\mu\colon X' \rightarrow X$ of $(X, D)$ such that if we write 
$(K_{\mathcal F'}+D')\vert_{D'} = K_{\mathcal F_{D'}}+{\rm Diff}_{D'}(\cal F)$ where 
$D' = \mu_*^{-1}D$, 
$\mathcal F' = \mu^{-1}\mathcal F$
and $\mathcal F_{D'}$ is the restricted foliation on $D'$, 
 then $Z'$ 
is not contained in the support 
of ${\rm Diff}_{D'}(\cal F')$, where $Z'\subset D'$ is the strict transform of $Z$ through the induced morphism $D'\to D$.
In particular, $\cal F'$ is smooth in a neighbourhood of the generic point of $Z'$. 
\end{lemma}
\begin{proof}
We may freely shrink about the generic point of $Z$, 
and so we may assume that 
$T_{\mathcal F}$ is locally free \cite[Corollary 1.4]{Hartshorne80}, 
and generated by $\partial_1, \dots, \partial_r$.

We will construct a model $\mu\colon X' \rightarrow X$ such that, using the same notation
as in the statement of the Lemma, 
there is a vector field $\partial \in T_{\mathcal F'}$ 
such that $\partial(I_{D'})$ is not contained in $I_{Z'}$. 
By (1) and  (2) of Proposition \ref{prop_explicit_calc} this will be our desired model.

Up to relabelling we may assume that $\partial_1 \in T_{\mathcal F}$
does not leave $D$ invariant.

%

%

Shrinking to a neighbourhood of the generic point of $Z$ we may find coordinates $(x, y, z_1, \dots, z_q) = (x, y, \underline{z})$
where $\dim X = q+2$
such that  
$D = \{x = 0\}, Z = \{x = y = 0\}$ and 
\[
\partial_1 = a(x, y, \underline{z})\partial_x+b(x, y, \underline{z})\partial_y+\sum_{j = 1}^q c_j(x, y, \underline{z})\partial_{z_j}
\]
where $a,b,c_1,\dots,c_q\in \mathcal O_{X, Z}$.
If $a$ is a unit in $\mathcal O_{X, Z}$ then we may take $X' = X$ and we are done, so we may assume that $a \in (x, y)$.
Since $D$ is not $\partial_1$-invariant we also see that $a \notin (x)$ and so (up to multiplying by a unit in $\mathcal O_{X, Z}$)
we may write  $a = y^k+x\alpha(x, y, \underline{z})$ where $k\ge 0$ and for some function $\alpha\in \mathcal O_{X, Z}$.  If $k = 0$, then $a$ is a unit and we are done. Thus, we assume that $k \ge 1$.
Note that $k$ is the order of tangency between $\partial_1$ and $D$ along $Z$.  We will show that after a finite sequence of 
blow ups we can reduce the order of tangency between the strict transform of $D$ and the pullback of the foliation $\mathcal F_1$ defined by $\partial_1$.
Assuming this can be done, then we may conclude by arguing by induction on $k$.

Consider the local chart of the blow up $b_1\colon X_1 \rightarrow X$ in $Z$ given by $x = x_1y_1, y = y_1, z_i = z_i$.  In this chart, the 
strict transform of $D$ is given by $\{x_1 = 0\}$ and the transform $Z_1$ of $Z$ (considered as subvariety of $D$) is given by $\{x_1 = y_1 = 0\}$.
The pullback of $\partial_1$ as a meromorphic vector field is 
\[[(y_1^{k-1}+x_1\alpha(x_1y_1, y_1, \underline{z})-\frac{x_1}{y_1}b(x_1y_1, y_1, \underline{z})]\partial_{x_1}+
b(x_1y_1, y_1, \underline{z})\partial_{y_1}+
\dots.\]


If the pullback of $\partial_1$ is in fact holomorphic, then in a neighbourhood of the transform $Z_1$ of $Z$, 
$b_1^{-1}\mathcal F$ is defined by a vector field of the form $(y_1^j+x_1\alpha_1)\partial_{x_1} +\dots$
where $j <k$ and for some function $\alpha_1\in \mathcal O_{X_1, Z_1}$.
Otherwise, $b_1^{-1}\mathcal F$ is defined by a vector field of the form 
\[
\delta = (y_1^k+x_1\alpha_1)\partial_{x_1}+y_1\beta_1\partial_{y_1}+\dots
\]
 for some functions $\alpha_1,\beta_1\in \mathcal O_{X_1, Z_1}$.  It is easy to check that if $b_2\colon X_2 \rightarrow X_1$ is the blow up
in $Z_1$, then the pullback of $\delta$ is holomorphic.  In particular, if we take the following coordinates on the blow up
$x_1 = x_2y_2, y_1 = y_2, z_i = z_i$, then in a neighbourhood of $Z_2$, the transform of $Z_1$, $b_2^{-1}b_1^{-1}\mathcal F_1$
is defined by a vector field of the form $(y_2^j+x_2\alpha_2)\partial_{x_2}+\dots$ where $j<k$ and for some function $\alpha_2\in \mathcal O_{X_2, Z_2}$
Thus, after at most two blow ups centred over $Z$ we see that we can reduce the order of tangency between the strict transform of $D$
and the pullback of $\mathcal F_1$, as required.
\end{proof}

\begin{theorem}
\label{thm_adj_sing}
Let $X$ be a normal variety, 
let $\mathcal F$ be a foliation of rank $r$  on $X$ and let $D$ be a prime divisor which is not 
$\mathcal F$-invariant.
Let $0 \leq \Delta = D+\Delta'$ be a $\mathbb Q$-divisor on $X$ such that  $D$ is not contained in the support of 
$\Delta'$.  Let $Z\subset X$ be a subvariety such that $Z\cap D$ is of codimension at least two in $D$ and such that $K_{\cal F}$ and $D$ are $\mathbb Q$-Cartier on $X\setminus Z$. 
Suppose that $(\mathcal F, \Delta)$
is canonical (resp. log canonical, resp. terminal, resp. log terminal).  Let $S \rightarrow D$ be the normalisation and let  $\Delta_S\coloneqq {\rm Diff}_S(\cal F,\Delta')$.

Then $(\mathcal F_S, \Delta_S)$
is canonical (resp. log canonical, resp. terminal, resp. log terminal).
\end{theorem}
\begin{proof}
Pick any divisor $E_S$ on some birational model $S' \rightarrow S$.
Let $\mu \colon Y \rightarrow X$ be any log resolution of $(X, \Delta)$ 
(we emphasise that this not necessarily a log
resolution of $\mathcal F$)  which extracts a divisor $E$ such that 
$E \cap S_Y = E_S$ where $S_Y \coloneqq \mu_*^{-1}D$. Let $\mathcal F_Y \coloneqq \mu^{-1}\mathcal F$ and 
$\Delta_Y \coloneqq \mu_*^{-1}\Delta$.

By Lemma \ref{lem_weak_resolution}, perhaps replacing $Y$ by a higher model we may assume
that if we write $(K_{\mathcal F_Y}+S_Y)\vert_{S_Y} \sim K_{\mathcal F_{S_Y}}+\Theta_{S_Y}$ 
then $\Theta_{S_Y}$
does not contain $E_S$ in its support.

By assumption $a(E, \mathcal F, \Delta) \geq 0$ (resp. $\geq -\epsilon(\mathcal F_Y, E)$, etc.)
and so we see that $a(E_S, \mathcal F_S, \Delta_S) \geq 0$ (resp. $\geq -\epsilon(\mathcal F_Y, E)$, etc.).
To conclude it suffices to show that if $a(E, \mathcal F, \Delta) < 0$ then 
$\epsilon(\mathcal F_{S_Y}, E_S) = 1$.  Suppose for sake of contradiction that $\epsilon(\mathcal F_{S_Y},E_S) = 0$.
By Lemma \ref{lem_weak_resolution} we see that $E_S \subset Y$ is not contained in 
$\sing\, {\mathcal F_Y}$
and so shrinking about a general point of $E_S$ we may assume that $\mathcal F_Y$ is smooth,
and that $\Delta_Y = S_Y$.  If $a(E, \mathcal F, \Delta)<0$ 
then a direct calculation shows that
the blow up of $Y$ along $E_S$ extracts an invariant divisor $F$ such that 
\[
a(F,\mathcal F,\Delta)=a(F,\mathcal F_Y, S_Y-a(E, \mathcal F, \Delta)E)<0,
\] 
contradicting the hypothesis that $(\mathcal F, \Delta)$ is log canonical. Thus, our claim follows.
\end{proof}

\begin{corollary}
\label{cor_diff_comparison}
Let $X$ be a potentially klt 
variety, let $\mathcal F$ 
be a foliation of rank $r$ on $X$ and let $D$ be a divisor which is not 
$\mathcal F$-invariant.
Let $0 \leq \Delta = D+\Delta'$ be a $\mathbb Q$-divisor on $X$ such that $D$ is not contained in the support of 
$\Delta'$. 
Suppose that $D$, $K_{\mathcal F}+\Delta$ and $K_X+\Delta$ 
are $\mathbb Q$-Cartier.
Let $n\colon S \rightarrow D$ be the normalisation and let ${\rm Diff}_S(X, \Delta')$ be the classical different associated to $(X, \Delta)$ on $S$.

Then we have an inequality of differents
\[
{\rm Diff}_S(\mathcal F, \Delta')_{{\rm n-inv}} \geq {\rm Diff}_S(X, \Delta')_{{\rm n-inv}}.
\]
\end{corollary}
\begin{proof}

To prove the  claim it suffices to work in the neighbourhood 
of the generic point of a 
 divisor $P \subset S$, which is not $\mathcal F_S$-invariant. 
Since $X$ is potentially klt, it has quotient singularities in a neighbourhood of the generic
point of $P$.  Arguing as in the proof of item  (1) of Proposition 
\ref{prop_explicit_calc} we may freely replace $X$ by a finite cover,
and so we may assume that $X$ is smooth.

Suppose that  $\pi\colon \overline{X} \rightarrow X$ is a log resolution of $(X, D+{\rm Supp}\, \Delta)$. 
 We may write 
\[
K_{\overline{\mathcal F}}+\overline{\Delta}+ \overline{D}+E = \pi^*(K_{\mathcal F}+\Delta+D)
\]
and
\[
K_{\overline{X}}+\overline{\Delta}+\overline{D}+E' = \pi^*(K_{\mathcal F}+\Delta+D)
\]
where $E, E'$ are $\pi$-exceptional, $\overline{\mathcal F} \coloneqq \pi^{-1}\mathcal F$, $\overline{D}\coloneqq \pi_*^{-1}D$
and $\overline{\Delta} \coloneqq \pi_*^{-1}\Delta$.
By \cite[Corollary 3.3]{Spicer20}  we have that $E'\le E$.   
Let $S$ be the normalisation of $D$ and let $\mu\colon \overline{ D} \rightarrow S$ be the induced morphism.

By Remark \ref{r_diff_b}  we have that  ${\rm Diff}_{\overline D}(\overline{\cal F},\overline{\Delta})\ge \overline \Delta|_{\overline D}$. Thus, if ${\rm Diff}_{\overline D}(\overline{X},\overline{\Delta})$ denotes the classical different on $\overline{X}$, then  
${\rm Diff}_{\overline D}(\overline{X},\overline{\Delta})= \overline \Delta|_{\overline D}$ and 
by Remark \ref{rmk_2} we have

\[
\begin{aligned}
 {\rm Diff}_S(\cal F,\Delta)&=
\mu_*({\rm Diff}_{\overline D}(\overline{\cal F},\overline{\Delta})+E|_{\overline{D}}) \\ 
&\ge \mu_*({\rm Diff}_{\overline D}(\overline{X},\overline{\Delta})+E'|_{\overline D})={\rm Diff}_S(X,\Delta).
\end{aligned}
\]
Thus, our claim follows. 
\end{proof}

\begin{lemma}
\label{lem_alg_int_disc}
Let $X$ be a potentially klt 
variety, let $\mathcal F$ 
be an algebraically integrable foliation of rank $r$ on $X$  which is induced by a morphism $f\colon X \rightarrow Z$
and let $D$ be a divisor which is not 
$\mathcal F$-invariant. 
Let $0 \leq \Delta = D+\Delta'$ be a $\mathbb Q$-divisor on $X$ such that $D$ is not contained in the support of 
$\Delta'$. 
Suppose that $D$, $K_{\mathcal F}+\Delta$ and $K_X+\Delta$ 
are $\mathbb Q$-Cartier.
Let $n\colon S \rightarrow D$ be the normalisation  and let ${\rm Diff}_S(X, \Delta')$ be the classical different associated to $(X, \Delta)$ on $S$.
Suppose that $(\mathcal F, \Delta)$ is log canonical.

Then 
\[
{\rm Diff}_S(\mathcal F, \Delta') = {\rm Diff}_S(X, \Delta')_{{\rm n-inv}}.
\]
\end{lemma}
\begin{proof}

By Corollary \ref{cor_diff_comparison} we have an inequality 
\[
{\rm Diff}_S(\mathcal F, \Delta') \geq {\rm Diff}_S(X, \Delta')_{{\rm n-inv}}.
\]
By Theorem \ref{thm_adj_sing} we have that 
$(\mathcal F_S, {\rm Diff}_S(\mathcal F, \Delta'))$
is log canonical and by Lemma \ref{rmk_1} it follows that ${\rm Diff}_S(\mathcal F, \Delta')$ has no $\mathcal F_S$-invariant components. 
Thus,
it suffices to show that  if $P \subset S$ is a divisor which is not $\mathcal F_{S}$-invariant then 
the coefficient of $P$ in each of the differents is the same.

Since $\mathcal F$ is induced by the morphism
$f\colon X \rightarrow Z$, $\mathcal F_{S}$ is induced by the restricted
morphism $(f\circ n)\colon S \rightarrow Z$.  Since $P$ is not $\mathcal F_{S}$-invariant,
$P$ dominates $Z$.
Let $\mu\colon X' \rightarrow X$ be any birational model and $E$
be any $\mu$-exceptional divisor centred on $n(P)$. 
Let us define 
\[
G \coloneqq (K_{\mu^{-1}\mathcal F}+\mu_*^{-1}\Delta-\mu^*(K_{\mathcal F}+\Delta) ) - 
(K_{X'}+\mu_*^{-1}\Delta-\mu^*(K_X+\Delta)).
\]
Note that $G$ is $\mu$-exceptional. 
In a neighbourhood of the generic fibre of $X' \rightarrow Z$ we have that 
$K_{\mu^{-1}\mathcal F} \equiv K_{X'}$, and so 
$G$
is $\mu$-numerically
trivial in a neighbourhood of the generic fibre of $X' \rightarrow Z$.  By the 
negativity lemma (cf. \cite[Lemma 3.39]{KM98}) in this neighbourhood $G = 0$.  
In particular, since $E$ dominates $Z$, we have
$a(E, \mathcal F, \Delta) = a(E, X, \Delta)$.  
By Remark \ref{rmk_2} 
we may then conclude that 
\[
m_P{\rm Diff}_S(\mathcal F, \Delta') = m_P{\rm Diff}_S(X, \Delta')
\]
as required.
\end{proof}

\begin{corollary}
\label{cor_diff_subfol}
Let $X$ be a potentially klt 
variety, let $\mathcal F$ 
be a foliation on $X$ and let $\mathcal H \subset \mathcal F$ be an algebraically integrable
subfoliation which is induced by a morphism and let $D$ be a divisor on $X$ which is not 
$\mathcal H$-invariant.
Let $0 \leq \Delta = D+\Delta'$ be a $\mathbb Q$-divisor on $X$ such that $D$ is not contained in the support of 
$\Delta'$. 
Suppose that $D$, $K_{\mathcal F}+\Delta$ and $K_X+\Delta$ 
are $\mathbb Q$-Cartier  and that  $(\mathcal H, \Delta)$ is log canonical.
Let $n\colon S \rightarrow D$ be the normalisation.  

Then 
\[
{\rm Diff}_S(\mathcal F, \Delta') \geq {\rm Diff}_S(\mathcal H, \Delta').
\]
\end{corollary}
\begin{proof}
We first remark that $D$ is not $\mathcal F$-invariant and if $P\subset S$ is a prime divisor which is not $\cal H_S$-invariant, then it is also not $\cal F_S$-invariant. Thus, the claim
is a direct consequence of Corollary \ref{cor_diff_comparison}
combined with
Lemma \ref{lem_alg_int_disc}.
\end{proof}

%
%
%
%
%
%
%
%

\subsection{Additional remarks}

\subsubsection{Failure of adjunction on singularities for invariant divisors}
\label{s_failure} Let $\cal F$ be a log canonical foliation on a smooth variety $X$ and let $D$ be a smooth $\cal F$-invariant divisor on $X$.
Then in general $(\mathcal F_D, {\rm Diff}_D(\cal F))$ 
is not log canonical, as the following
example shows:

\begin{example}
Let $X = \mathbb A^3$ with coordinates $x, y, z$ and let $\mathcal F$ be the rank one foliation on $X$ defined by the vector field
\[
\partial \coloneqq x^2\partial_x+y^2\partial_y+z\partial_z.
\]
By \cite[Fact I.ii.4]{McQP09} $\mathcal F$ has log canonical singularities since 
its semi-simple part $z\partial_z$
is non-zero.
  Set $D = \{z = 0\}$.  Then $D$ is $\cal F$-invariant and $\mathcal F_D$ is generated by $x^2\partial_x+y^2\partial_y$
whose semi-simple part is zero and by \cite[Fact I.ii.4]{McQP09} again, it is  not log canonical.
\end{example}

\subsubsection{Failure of Bertini's theorem}

Let $X$ be a smooth variety and let $\mathcal F$ be
a foliation with canonical singularities.  Let $A$ be an ample Cartier divisor.
In general it is not the case that we may find $0 \leq D \sim_{\mathbb Q} A$ 
such that $(\mathcal F, D)$ is log canonical, as the following example shows (see also \cite[pag. 12]{ACSS}). Let $\pi\colon X_0:=\mathbb P^2\times\mathbb P^1\to \mathbb P^2$ be the projection, let $\ell \subset X_0$ be a smooth curve such that $\pi(\ell)$ is a line and let $X$ be the blow up of $X_0$ along $\ell$. Then the foliation $\cal F$ induced by the morphism $X\to \mathbb P^2$ admits canonical singularities, but $(\cal F,D)$ is not log canonical for any ample $\mathbb Q$-divisor $D\ge 0$. 

 Moreover, the same example shows that it is in general not 
possible to choose $D$ to be reduced and irreducible and so that ${\rm Diff}_S(\cal F) = 0$, where $S\to D$ is the normalisation. 

\subsubsection{Other definitions of foliation}
\label{s_otherdef}
Occasionally 
in the literature a foliation is defined to be a quotient of the cotangent sheaf.
From the perspective of defining an adjunction formula (especially on singular varieties) 
our definition seems to be more appropriate.

The $S_2$ condition likewise seems to be important. 
On non-normal varieties there exist quotients of the cotangent sheaf which do not 
seem to correspond to reasonable foliations.  Consider the following example.
Let $X\coloneqq \{x = y = 0\} \cup \{z = w = 0\} \subset \mathbb A^4$, let $\omega$ be 
the restriction to $X$ of the $1$-form $dz+xdy+ydx$ and consider
the quotient
$\Omega^1_X \rightarrow \Omega^1_X/(\omega)$.  There is no way to lift this quotient to a rank one
foliation on a neighbourhood of $(0, 0, 0, 0)$,
 however, for a foliation one would expect that this 
should always be possible.


\section{Cone theorem for rank one foliated pairs}
\label{section_cone}
Throughout this section we assume that all our varieties are defined over $\mathbb C$.

The cone theorem for rank one foliations was intially considered in \cite{bm16}, and a refined version 
has appeared in \cite{McQ05}.  The goal of this section is to present a more general version which is more suitable to run the MMP. 

In order to study curves which are contained in the singular locus of a foliation we first need some  preliminary results. 
Indeed, the definition of the singular locus of a rank one foliation presented in \cite{McQ05} is slightly different 
than the one given above.  We recall McQuillan's definition now.

Let $X$ be a normal variety, let $\mathcal F$ be a rank one foliation on $X$ such that 
$K_{\cal F}$ is $\mathbb Q$-Cartier.  Let $x \in X$ be a point and let $U$ be an open neighbourhood of $x$.  
Up to replacing $U$ by a smaller neighbourhood we may find an index one cover $\sigma\colon U' \rightarrow U$ associated to $K_{\mathcal F}$
and such that $\sigma^{-1}\mathcal F$ is generated by a vector field $\partial$.

We say that $\mathcal F$ is {\bf singular in the sense of McQuillan} at $x \in X$ provided
there exists an embedding $U' \rightarrow M$ where $M$ is a smooth variety and a lift $\tilde{\partial}$
of $\partial$ to a vector field on $M$  such that $\tilde{\partial}$ vanishes at $\sigma^{-1}(x)$.
We denote by $\sing^+{\mathcal F}$ the locus of points $x \in X$ where $\mathcal F$ is singular in the sense 
of McQuillan. Note that $\sing^+{\mathcal F}$ does not depend on the choice of $U'$ and it is a closed subset of $X$. 

On a smooth variety it is easy to see that $\sing^+{\mathcal F} = \sing{\mathcal F}$, but in general
it is unclear if these two notions of singularity agree.  We do, however, have the following inclusion of
singular loci:

\begin{lemma}\label{l_sing}
Let $X$ be a normal variety, let $\mathcal F$ be a rank one foliation on $X$ such that 
$K_{\cal F}$ is $\mathbb Q$-Cartier.

Then  $\sing \mathcal F \subset \sing^+{\mathcal F}$. 
\end{lemma}
\begin{proof}
By \cite[Corollary 5.14]{Druel21} we may freely replace $X$ by an index one cover of $K_{\mathcal F}$ 
and so we may assume that $K_{\mathcal F}$ is Cartier, in which case the Lemma is easy.
\end{proof}

\begin{lemma}
\label{lem_invariant_cover}
Let $X$ be a normal variety, let $\mathcal F$ be a rank one foliation on $X$ such that 
$K_{\cal F}$ is $\mathbb Q$-Cartier. 
Let $V \subset X$ be an irreducible subvariety such that 
$V$ is not contained in $\sing^+\, \mathcal F$.  Assume that there exists an index one cover 
$\sigma\colon X' \rightarrow X$ associated to $K_{\mathcal F}$ and let $\cal F'\coloneqq \sigma^{-1}\cal F$.

Then $V$ is $\mathcal F$-invariant
if and only if $V'\coloneqq \sigma^{-1}(V)$ is $\mathcal F'$-invariant.
Moreover, if $V$ is invariant and  
if we denote by $\mathcal F_V$ (resp. $\mathcal F'_{V'}$) the restricted
foliation on $V$ (resp. $V'$), then $(\sigma|_{V'})^{-1} \mathcal F_V = \mathcal F'_{V'}$.
\end{lemma}
\begin{proof}
We first remark that if $\sigma$ is \'etale in a neighbourhood
of the generic point of $V$, then both points of the lemma are clear.

Suppose that $V$ is $\mathcal F$-invariant.  In this case, by definition it follows that 
$K_{\mathcal F}$ is Cartier in a neighbourhood of the 
generic point of $V$ and, therefore,
$\sigma$ is \'etale in a neighbourhood of the generic point 
of $V$, in which case we may conclude. 

Suppose that $\sigma^{-1}(V)$ is $\mathcal F'$-invariant.  
Let $x \in \sigma^{-1}(V)$ be a general point.  By \cite[Lemma I.2.1]{bm16} 
there exists an analytic neighbourhood $U$ of $x$
and a holomorphic submersion $p\colon U \rightarrow Z$ such that $T_{\mathcal F'}\vert_U = T_{U/Z}$.
After possibly shrinking $U$, we may assume that $\sigma$ is a Galois cover with Galois cover $G=\mathbb Z/m\mathbb Z$. Since the fibre of $p$ are invariant with respect to $G$, it follows that the action of $G$ descends to $Z$. 

After possibly shrinking $U$ further, we may write $U \cong \mathbb D \times Z$ where 
$\mathbb D$ is the unit disc and $p$
is given by projection onto the second coordinate.  Moreover, we may assume that $G$ acts on $\mathbb D$ so that the diagonal action of $G$ on $\mathbb D \times Z$ coincides with the action on $U$. 
In particular, if $t$ is a coordinate on $\mathbb D$, then 
$T_{\mathcal F'}$ is defined by $\partial_t$ and 
the Galois group acts on $t$ by $t \mapsto \xi^{-1} t$ where $\xi$ is a primitive $m$-th root of unity.
It follows
that the ramification locus of $\sigma$ is of the form $\{t = f_1 = \dots= f_r = 0\}$, 
for some functions $f_1,\dots,f_r\in \mathcal O_X$.

Since $\cal F'$ is defined by $\partial_{t}$, any $\mathcal F'$-invariant variety is locally of the form $p^{-1}(W)$ 
where $W \subset Z$ is a subvariety. Thus, no invariant subvariety is contained
in the ramification locus of $\sigma$, and so $\sigma$ is \'etale at the generic point of $V$
and we may conclude.
\end{proof}

\begin{lemma}\label{l_terminal}
Let $X$ be a normal variety, let $\mathcal F$ be a rank one foliation on $X$ such that 
$K_{\cal F}$ is $\mathbb Q$-Cartier.
Let $Z \subset X$ be a subvariety which is not $\cal F$-invariant and  is not contained in $\sing^+\, \mathcal F$.

Then $\cal F$ is terminal at the generic point of $Z$. 
\end{lemma}

\begin{proof}
The problem is local about a general point $z \in Z$,
so we are free to shrink about a general point of $Z$.  
By \cite[Corollary III.i.5]{McQP09}
and Lemma \ref{lem_invariant_cover}, we may therefore 
replace $X$ by the index one cover associated to $K_{\mathcal F}$, and so
we may assume that $K_{\mathcal F}$ is Cartier.
The result then follows from \cite[Fact III.i.1]{McQP09}.
\end{proof}

\begin{lemma}
\label{l_adjunction}
Let $p\colon X\to Y$ be a smooth morphism between normal analytic varieties 
and let $\cal F$ be the foliation on $X$ induced by $p$. 
Let $\Delta\ge 0$ be a $\mathbb Q$-Cartier $\mathbb Q$-divisor on $X$. Let $0\in Y$ and let $D_0\coloneqq p^{-1}(0)$. 

Then $(\cal F,\Delta)$ is log canonical in a neighbourhood of $D_0$ if and only if $D_0$ 
is not contained in the support of $\Delta$ and $(D_0,\Delta_0)$ is log canonical, where $\Delta_0\coloneqq \Delta|_{D_0}$. 
\end{lemma}
\begin{proof}
Note that if $D_0$ is contained in the support of $\Delta$, $\dim Y \ge 2$ and $q\colon \overline X\to X$ is the blow up of $X$ along $D_0$ 
with exceptional divisor $E$ then $E$ is $q^{-1}\cal F$-invariant and $a(E,\cal F,\Delta)<0=\epsilon(E)$. 
Thus, $(\cal F,\Delta)$ is not log canonical around $D_0$.  If $\dim Y = 1$, and $D_0$ is contained in the support of $\Delta$ then 
$(\mathcal F, \Delta)$ is not log canonical
by Lemma \ref{rmk_1}. Therefore we may assume that $D_0$ is not contained in the support of $\Delta$.

Let $\beta\colon Y'\to Y$ be a resolution of $Y$ and let 
$X'\coloneqq X\times_{Y} Y'$. 
Let $\alpha\colon X'\to X$  and $p'\colon X'\to Y$ be the induced morphisms. 
Then $p'$ is a smooth morphism and if $\cal F'$ is the foliation induced by $p'$, it follows that $\cal F'=\alpha^{-1}\cal F$. 
Let $\Delta'\coloneqq \alpha_*^{-1}\Delta$. Then $K_{\cal F'}+\Delta'=\alpha^*(K_{\cal F}+\Delta)$.

For any $y\in \beta^{-1}(0)$, let $D_y\coloneqq p'^{-1}(y)$ be the corresponding fibre. 
Note that $D_y$ is not contained in the support of $\Delta'$ and if $\Delta_y\coloneqq \Delta'|_{D_y}$ then $(D_y,\Delta_y)\cong (D_0,\Delta_0)$.  

Pick $y\in \beta^{-1}(0)$ and let $\Sigma$ be a reduced divisor on $Y'$ such that $(Y',\Sigma)$ is log smooth and $y$ is 
a zero-dimensional stratum of $\Sigma$. We claim that  $(\cal F',\Delta')$ is log canonical in a neighbourhood of $D_y$ if 
and only if $(X',\Delta'+p'^*\Sigma)$ is. 
Indeed, if $\gamma\colon X''\to X'$ is a birational morphism then, as in the proof of \cite[Lemma 3.1]{ACSS}, we have that
\[
K_{X''}+\gamma_*^{-1}(p'^*\Sigma)=\gamma^*(K_{X'}+p'^*\Sigma)+F
\]
and 
\[
K_{\gamma^{-1}\mathcal F'} = \gamma^*K_{\cal F'}+G
\]
where $F,G$ are $\gamma$-exceptional divisors such that 
$G=F+\sum(1-\epsilon(E))E$, where the sum runs over all the $\gamma$-exceptional prime divisors. Thus, our claim follows.

By inductively applying inversion of adjunction (see \cite[Theorem 1.1]{MR4702587}), we have that $(X',\Delta'+p'^*\Sigma)$ is log canonical in a neighbourhood of 
$D_y$ if and only if $(D_y,\Delta_y)$ is log canonical. Thus, we have that $(\cal F,\Delta)$ is log canonical (resp. log terminal) in 
a neighbourhood of $D_0$ if and only if  $(D_0,\Delta_0)$ is log canonical (resp. log terminal).\end{proof}

\begin{lemma}
\label{lem_adj_smooth}
Let $X$ be a normal variety and let $\mathcal F$ be a rank one foliation
on $X$ such that $K_{\mathcal F}$ is $\mathbb Q$-Cartier.
Let $Z \subset X$ be a subvariety which is not $\cal F$-invariant and it is not contained in $\sing^+\, \mathcal F$.
Let $W$ be a (possibly analytic) 
$\mathcal F$-invariant subvariety which contains $Z$ with $\dim W = \dim Z+1$.
Let $n\colon S \rightarrow W$ be the normalisation and let $\Delta \geq 0$ be a $\mathbb Q$-Cartier $\mathbb Q$-divisor on $X$ which does not contain $W$ in its support.

Then
\begin{enumerate}
\item 
there exists $\lambda>0$
such that $(\mathcal F, \lambda \Delta)$ is terminal at the generic point of $Z$;

\item if $({\mathcal F},\Delta)$ is log canonical  
 in a neighbourhood of the generic point of $Z$ and 
 $\mathcal F_S$ is the restricted foliation, 
then $(\mathcal F_S, {\rm Diff}_S(\cal F,\Delta))$ is log canonical 
 in a neighbourhood of 
the generic point of $n^{-1}(Z)$.
\end{enumerate}
\end{lemma}
\begin{proof}
The problem is local about a general point $z \in Z$,
so we are free to shrink about a general point of $Z$.  
By \cite[Corollary III.i.5]{McQP09}
and Lemma \ref{lem_invariant_cover}, we may therefore 
replace $X$ by the index one cover associated to $K_{\mathcal F}$, and so
we may assume that $K_{\mathcal F}$ is Cartier.

We may also assume that $Z$ does not intersect $\sing^+\, \mathcal F$
 and so by \cite[Lemma 1.2.1]{bm16} 
there exists an analytic neighbourhood $U$ of $z$ 
and a smooth morphism $p\colon U \rightarrow V$ such that $\cal F|_U$ is the foliation induced by $p$.
If  $v \in V$ is a general closed point, set $L_v \coloneqq p^{-1}(v)$.
Thus, after possibly shrinking again, we may assume that $Z$ and $W$ are smooth. In particular, $n$ is an isomorphism. 

Since  $W$ is not contained in the support of $\Delta$, it follows that  for a general closed point $v \in p(Z)$, $L_v$ is not contained in the support of $\Delta$. Since $p$ is smooth, there exists $\mu >0$ such that 
 $(L_v, \mu \Delta|_{L_v})$ is log canonical. Thus,
Lemma \ref{l_adjunction} implies that $(\mathcal F, \mu \Delta)$ is log canonical in a neighbourhood of $L_v$. 
Since $v$ is general it follows that $(\mathcal F, \mu \Delta)$ is log canonical at the generic point of $Z$.  
By Lemma \ref{l_terminal},
we have that  $\mathcal F$
is terminal at the generic point of $Z$, and so $(\mathcal F, \frac{\mu}{2}\Delta)$ is log terminal
at the generic point of $Z$. By \cite[Corollary III.i.4]{McQP09}, we have that 
for any birational morphism $\pi\colon X' \rightarrow X$, 
the $\pi$-exceptional locus is $(\pi^{-1}\mathcal F)$-invariant,  and so in fact 
$(\mathcal F, \frac{\mu}{2} \Delta)$ is terminal at the generic point of $Z$. Thus, (1) follows. 

We now prove Item (2). 
For a general closed point $v \in p(Z)$, $L_v$ is not contained in $\Delta$ and so may apply Lemma \ref{l_adjunction} to see that
if $({\mathcal F},\Delta)$ is log canonical then  $(L_v, \Delta\vert_{L_v})$
is log canonical. 
Let $\cal F_{W}$ be the restricted foliation  on $W$, whose existence is guaranteed 
by Proposition-Definition \ref{prop_inv_subadj_analytic}
 Since $W$ is smooth, we have that 
$ {\rm Diff}_W(\cal F,\Delta)|_{L_v}=\Delta|_{L_v}$. Thus, 
 we may again apply Lemma \ref{l_adjunction}
to see that $(\mathcal F_W, {\rm Diff}_W(\cal F,\Delta))$ is log canonical in a neighbourhood of $L_v$.  Since $v$ is general 
we see that $(\mathcal F_W, {\rm Diff}_W(\cal F,\Delta))$ is log canonical in a neighbourhood of the generic point of $Z$.
\end{proof}

\begin{lemma}
\label{lem_produce_germ}
Let $X$ be a normal variety, let $\mathcal F$ be a rank one foliation on $X$ 
such that $K_{\mathcal F}$
is $\mathbb Q$-Cartier. 
Let $C \subset X$ be a curve which is not $\cal F$-invariant
and is not contained in $\sing^+\, \cal F$.

Then there exists a birational morphism $p\colon X' \rightarrow X$ such that  
if $\cal F' \coloneqq p^{-1}\cal F$ 
 then the following hold:
\begin{enumerate}
\item
$K_{\cal F'}+ E = p^*K_{\cal F}$ for some $p$-exceptional 
$\mathbb Q$-divisor $E\geq 0$; 

\item $p^{-1}$ is an isomorphism at the general point of $C$ and  if  $C'$ is the strict transform of $C$ in $X'$ then 
 $\cal F'$ is terminal at all points $P \in C'$; and
 
 \item after possibly replacing $X$ by an analytic neighbourhood of 
 $C$, there exist a  $\cal F'$-invariant surface 
 $\Gamma$ containing $C'$.
\end{enumerate}
\end{lemma}
\begin{proof} 
Since $C$ is not $\cal F$-invariant and
$C$ is not contained in $\sing^+\, \cal F$, by Lemma \ref{l_terminal}
 we have that 
$\cal F$ is terminal at the generic point of $C$.
Let $P_1, ..., P_k\in C$ be all the closed points at which $\cal F$ is not terminal. 
Let $H$ be a sufficiently general ample divisor on $X$  such that $P_1,\dots,P_k$ are not contained in its support. 
We may assume that 
$\mathcal O_X(mK_{\mathcal F})\vert_{X\setminus H} \cong \mathcal O_{X\setminus H}$
where $m>0$ is the Cartier index of $K_{\mathcal F}$.
Thus, we may find a finite Galois morphism 
$\sigma\colon Y \rightarrow X$ with Galois group $G$,
 which is quasi-\'etale outside
 $H$ and such that $\sigma^*K_{\cal F}$ is Cartier around any closed point $Q\in \widetilde C$ 
such that $\sigma(Q)\in  \{P_1,\dots,P_k\}$. 
 Let $\cal G \coloneqq \sigma^{-1}\cal F$ and let $\widetilde C\coloneqq \sigma^{-1}(C)$. 
 Then  
 $K_{\cal G}$ is Cartier around any closed point $Q\in \widetilde C$ 
such that $\sigma(Q)\in  \{P_1,\dots,P_k\}$
 and 
 $\cal G$ is terminal at all closed points $Q\in \widetilde C$
such that $\sigma(Q)\notin \{P_1,\dots,P_k\}$. 

By \cite[Proposition I.2.4]{bm16}, there exists a 
birational morphism $q\colon Y' \rightarrow Y$, obtained as a sequence of $G$-equivariant 
blow ups in $\cal G$-invariant points,  
such that in a neighbourhood of $q^{-1}(\sigma^{-1}(\{P_1, \dots, P_k\}))$ the strict transform $\widetilde C'$ of $\widetilde C$ in $Y'$ is 
disjoint from $\sing^+ q^{-1}\cal G$.
Moreover, 
\cite[Proposition I.2.4]{bm16}
implies that  $K_{q^{-1}\cal G} + E'= q^*K_{\cal G}$ where
$E' \geq 0$ is $q$-exceptional. 

We have that $G$ acts on $Y'$ and $q$  is $G$-invariant. Thus, 
if $X' \coloneqq Y'/G$ then there exists a birational morphism
$p\colon X' \rightarrow X$.  
It  follows by  \cite[Lemma 3.4]{Druel21} 
that $K_{\cal F'} +E = p^*K_{\cal F}$
where $\cal F' = p^{-1}\cal F$ and $E \geq 0$ is $p$-exceptional (see also \cite[Proposition 2.2]{CS20}). Thus, (1) follows. 

Notice  that $p^{-1}$ is an isomorphism outside $P_1,\dots,P_k$ and, in particular, $\cal F'$ is terminal at all points in $C'\setminus p^{-1}(\{P_1,\dots,P_k\})$. 
We now show that 
$\mathcal F '$
is terminal at all points of $C'$. By  \cite[Corollary III.i.4]{McQP09}, we have that exceptional divisor over $Y'$ and whose centre is in $\widetilde C' \cap q^{-1}(\sigma^{-1}(\{P_1, \dots, P_k\}))$,
 is invariant. Thus,  every exceptional divisor over $X'$ and whose centre is in $ C' \cap p^{-1}(\{P_1, \dots, P_k\}))$ is also invariant.  As in the proof of 
 \cite[Lemma 4.3]{Druel17} (see also \cite[Corollary III.i.5]{McQP09}), 
 we have that $\cal F'$ is terminal at any  point in $C'\cap p^{-1}(\{P_1,\dots,P_k\})$. Thus, (2) follows. 


By \cite[Proposition I.2.4]{bm16} there also exists a $q^{-1}\mathcal G$-invariant 
surface germ $S$ containing $q_*^{-1}\widetilde{C}$. We may take $\Gamma = r(S)$
where $r\colon Y' \rightarrow X'$ is the quotient map
and so (3) follows.
%
\end{proof}

\begin{lemma}
\label{lem_nef_on_sing}
Let $X$ be a normal 
variety and let $(\cal F, \Delta)$ be a  rank one foliated pair on $X$ such that 
$K_{\cal F}$ and $\Delta$ are $\mathbb Q$-Cartier and $\Delta\ge 0$.
Let $C \subset\sing^+\, \cal F$ be a curve
and suppose that $(\cal F, \Delta)$ is log canonical
at the generic point of $C$.

Then $(K_{\cal F}+\Delta)\cdot C\geq 0$.
\end{lemma}
\begin{proof}
First suppose that $\Delta = 0$.
When $\cal F$ is Gorenstien this is proven in \cite[\S 4.1]{bm16},
the general case where $\cal F$ is only $\bb Q$-Gorenstein is proven
in \cite[Fact II.d.3]{McQ05}. 

Assume now that $\Delta\neq 0$. 
We claim that $C$ is not contained in ${\rm Supp}\, \Delta$. 
Supposing the claim we have that $\Delta\cdot C \geq 0$
and by the previous case it follows that $(K_{\cal F}+\Delta)\cdot C \geq 0$.

We now prove the  claim.  We may replace $X$
by a neighbourhood of the generic point of $C$ and so we may assume that the index 
one cover associated to $K_{\mathcal F}$ exists. 
By \cite[Corollary III.i.5]{McQP09} and by the definition of $\sing^+\, \mathcal F$, we may also freely replace $X$ by this index 
one cover, and so we may freely assume that $K_{\mathcal F}$ is Cartier and so $\mathcal F$ 
is generated by a vector field $\partial$.
Since $C \subset \sing^+\, \mathcal F$, it follows that $C$ is $\mathcal F$-invariant. Thus,  \cite[Lemma 1.1.3]{bm16}  implies that 
 the blow up $b\colon X' \rightarrow X$ along $C$ extracts 
a divisor $E$ of discrepancy $\leq 0$. Since $E$ is also $b^{-1}\cal F$-invariant 
it follows that 
the discrepancy of $E$ is
in fact $\leq -\epsilon(E)$.  Since $(\mathcal F, \Delta)$
is log canonical at the generic point of $C$, 
we may then conclude that $C$ is not contained
in ${\rm Supp}\, \Delta$ and our claim follows.
\end{proof}

\begin{theorem}
\label{thm_cone}
Let $X$ be a normal projective
 variety, let $\cal F$ be a rank one foliation on $X$
and let $\Delta \geq 0$ be a $\mathbb Q$-divisor on $X$ such that $K_{\cal F}$ and $\Delta$ are $\mathbb Q$-Cartier.

Then there are $\mathcal F$-invariant rational curves $C_1,C_2,\dots$ on $X$ not contained in ${\sing^+ }\, \mathcal F$
 such that 
\[
0<-(K_{\cal F}+\Delta)\cdot C_i\le 2\dim X\]
and
$$\overline{\rm NE}(X)=\overline{\rm NE}(X)_{K_{\cal F}+\Delta\ge 0}+Z_{-\infty}+
\sum_i \mathbb R_+[C_i]$$
where $Z_{-\infty}\subset \overline{\rm NE}(X)$ is a closed subcone contained in the span of the images of
$\overline{\rm NE}(T) \rightarrow \overline{\rm NE}(X)$ of non-log canonical centres $T \subset X$
 of $(\cal F, \Delta)$.
\end{theorem}

\begin{proof}
Let $R \subset \overline{NE}(X)$ be a $(K_{\cal F}+\Delta)$-negative exposed extremal ray (e.g. see \cite[Definition 6.1]{Spicer20})
and let $H_R$ be a supporting hyperplane to $R$. After possibly rescaling $H_R$ we may assume that there exists an ample $\mathbb R$-divisor  $H$  such that $H_R\sim_{\mathbb R}K_{\cal F}+\Delta +H$.  
Let $W$ be a component of the null locus ${\rm Null} \, H_R$ of $H_R$ 
(e.g. see \cite[Definition 1.3]{birkar17}) and let $n\colon S\to W$ be its normalisation. 
Thus,  $H_R|_S$ is not big and by Nakamaye's Theorem \cite[Theorem 1.4]{birkar17}, 
it follows that $W$ is a component of the stable base locus of $H_R-A$ for any 
sufficiently small ample $\mathbb R$-divisor $A$ on $X$. 
We want to show that,  after possibly replacing $W$ by another irreducible component,  $R$ is contained in the image of 
$\overline{\rm NE}(W) \rightarrow \overline{\rm NE}(X)$.
Indeed, first note that we can take $A$ so that $H_R-A$ is a $\mathbb Q$-divisor. 
Since $H_R\cdot \xi=0$ for all $\xi$ such that $[\xi]\in R$, we may find a sequence of classes $[\xi_i]\in {\rm NE}(X)$ such that
$(H_R-A)\cdot \xi_i<0$ and whose limit is a non zero element $[\xi_0]\in R$. Thus, after possibly taking a subsequence and after possibly replacing $\xi_i$ by an irreducible curve, we may assume that each $\xi_i$ is contained in the same connected component of the stable base locus of $H_R-A$.  Thus, our claim follows. 

By  \cite[Definition/Summary I.ii.6]{McQP09}, the union of all the non-log canonical centres of $(\cal F,\Delta)$ is closed in $X$. 
If $(\cal F,\Delta)$ is not log canonical at the generic point of $W$ then $R\subset Z_{-\infty}$ and we are done.
Thus, we may assume that  $(\cal F,\Delta)$ is log canonical at the generic point of $W$.

\medskip

We now distinguish three cases. Suppose first that $W$ is contained in $\sing^+\, \mathcal F$. 
In this case, Lemma \ref{lem_nef_on_sing} implies that every curve in $W$ which is 
$(K_{\mathcal F}+\Delta)$-negative 
is contained in a non-log canonical centre of $(\mathcal F, \Delta)$. 
Thus, there exists a subvariety $T\subset W$ which is a non-log canonical centre of $(\cal F,\Delta)$ and such
that $R$ is contained in the image of $\overline{\rm NE}(T) \rightarrow \overline{\rm NE}(X)$ and  we may conclude. 

\medskip 

Now suppose that $W$ is not contained in $\sing^+\, \mathcal F$ and is $\mathcal F$-invariant. In particular,  \cite[Lemma I.1.3]{bm16} implies that  $W$ is a log canonical centre for $\cal F$. Thus, since $(\cal F,\Delta)$ is log canonical at the generic point of $W$, it follows that $W$ is not contained in the support of $\Delta$.  
Let $\cal F_{S}$ be the restricted foliation  on $S$, whose existence is guaranteed 
by Proposition-Definition \ref{prop_inv_subadj_analytic} and let $\Delta_{S}\coloneqq {\rm Diff}_S(\cal F,\Delta)$. 
Since $H_R|_S$ is not big, we may apply 
\cite[Corollary 2.28]{Spicer20}
to conclude that $S$ is covered by $(K_{\cal F_S}+\Delta_S)$-negative rational curves $C$
tangent to $\cal F$ which span $R$ 
and such that 
\[0<-(K_{\cal F_{S}}+\Delta_{S})\cdot C \le 2\dim S,
\]
(e.g. see the proof of \cite[Theorem 6.3]{Spicer20} for more details). Thus, we may easily conclude.

\medskip

Finally suppose that $W$ is not contained in ${\sing}^+\, \mathcal F$ and is not $\mathcal F$-invariant. 
We show that this case does not in fact occur by showing that it leads to a contradiction. 
In this case, $W$ is a proper subvariety of $X$ and $H_R$ is big. 
Let $0<\epsilon\ll 1$ be a rational  number such that $H_R-\epsilon H$ is big and ${\rm Null}\, H_R$ coincides with  the stable base locus of $H_R-\epsilon H$. Let $G\sim_{\mathbb Q} H_R-\epsilon H$ be an effective $\mathbb Q$-divisor. Since $H_R|_S$ is not big and $H$ is ample, there exists a curve $C\subset S$ passing through the general point of $S$ and such that  $G\cdot C<0$. Moreover, we have  
\[
(K_{\mathcal F}+ \Delta)\cdot C = (H_R-H)\cdot C=\left( G - (1-\epsilon)H\right)\cdot C<0. 
\]
 
Note that, from now on, we no longer require   $K_{\cal F}$ and $\Delta$ to be $\mathbb Q$-Cartier, but just that $K_{\cal F}+\Delta$ is $\mathbb Q$-Cartier. Thus, we may replace $X$ by a model guaranteed to exist by 
Lemma \ref{lem_produce_germ}, and so we may find a germ of a $\cal F$-invariant surface $\Gamma$ containing $C$. 

Since $W$ is not $\cal F$-invariant and $\Gamma$ intersect the general point of $W$, it follows that $\Gamma$ is not contained in $W$.  In particular, since $W$ is a component of the stable base locus of $H_R-\epsilon H$, after possibly replacing $G$ by an effective $\mathbb Q$-divisor $G'$ which is $\mathbb Q$-linearly equivalent to  $G$, we may assume that $\Gamma$ is not contained in the support of $G$. 
Let $\nu\colon Y\to \Gamma$ be the normalisation.
By Proposition-Definition \ref{prop_inv_subadj_analytic}, if $\cal F_Y$ is the foliation induced on $Y$ then we may write 
\[
(K_{\mathcal F}+  \Delta)\vert_Y \sim_{\mathbb Q}  K_{\mathcal F_Y}+\Delta_Y
\]
where $\Delta_Y\coloneqq {\rm Diff}_Y(\cal F,\Delta)$.
By Lemma \ref{lem_adj_smooth}, we have that 
$(\mathcal F_Y,  \Delta_Y)$ is log canonical at the generic point of $C$ and, in particular, $m_C\Delta_Y\le 1$. Let $t\ge 0$ such that $\Delta_Y+tC=\Theta+C$ where $\Theta\ge 0$ is a $\mathbb Q$-divisor whose support does not contain $C$. Let $\mu>0$ be a rational number such that $m_C(\mu G|_Y)=1$. Then, considering $C$ as a curve in $Y$, we have $C^2\le C\cdot \mu G|_Y<0$ and so 
\[
(K_{{\mathcal F}_Y}+\Theta+C)\cdot C=(K_{\cal F}+\Delta)\cdot C+tC^2<0.
\]
On the other hand, by   \cite[Proposition 3.4]{Spicer20} (note that the proposition is stated to hold for projective varieties, but it works equally well in our context) and since the restricted foliation on $C$ is the foliation by points, we have that $(K_{{\mathcal F}_Y}+\Theta+C)\cdot C\ge 0$, a contradiction. 
Thus, our result follows. 
\end{proof}

Using the same methods as in the last case of the proof above, we have the following: 
\begin{corollary}
	\label{cor_null}
Let $X$ be a normal projective variety, let $\cal F$ be a rank one foliation on $X$ and let $\Delta\ge 0$ be a $\mathbb Q$-Cartier divisor such that $
K_{\cal F}$ and $\Delta$ are $\mathbb Q$-Cartier. Assume that 
$K_{\cal F}+\Delta$ is nef and
$(\cal F,\Delta)$ is log canonical away from a finite set of  closed points. 

Then ${\rm Null} \, (K_{\cal F}+\Delta)$ is a union of $\cal F$-invariant subvarieties, and subvarieties contained in $\sing^+\, \mathcal F$. 
\end{corollary}

In \cite[Theorem 1.2]{CouPer06} a dynamical characterisation of ample line bundles on smooth 
surfaces was provided.
As a consequence of the above theorem we are able to extend this to higher dimensions:

\begin{corollary}
\label{cor_ampleness}
Let $X$ be a normal projective variety and let $L$ be a $\mathbb Q$-Cartier divisor. 
Suppose that
\begin{enumerate}
\item $L^{\dim X} \neq 0$; 

\item for some $q >0$ there exists a rank one foliation $\mathcal F$ such that $K_{\cal F}$ is $\mathbb Q$-Cartier and  
$K_{\mathcal F} \equiv qL$; and 

\item $\dim \sing^+ \, \mathcal F\le 0$ and $\mathcal F$ admits no
invariant positive dimensional subvarieties.
\end{enumerate}

Then $L$ is ample.
\end{corollary}
\begin{proof}
By (3) and by Theorem \ref{thm_cone}, it follows that $L$ is nef and so by (1) we have that $L^{\dim X}>0$, and hence $L$ is big.
	By Lemma \ref{l_terminal}, $\cal F$ is log canonical outside $ \sing^+ \, \mathcal F$ and by Corollary \ref{cor_null} it follows that ${\rm Null}~ K_{\mathcal F} = \emptyset$ and so by 
	the Nakai-Moishezon criterion for ampleness we see that $K_{\mathcal F}$ is ample.
%
\end{proof}

\section{Family of leaves of an algebraically integrable foliation}

Let $X$ be a normal projective variety and let 
$\mathcal F$ be an algebraically integrable foliation
on $X$.  The following construction follows from \cite[Lemma 3.2]{AD13}:

\begin{construction}
\label{c_fam} There exists a closed subvariety $Z \subset {\rm Chow}(X)$
whose general point parametrises the closure of a leaf of $X$.
Let $p_X\colon X\times Z \rightarrow X$ and $p_Z \colon X \times Z \rightarrow Z$
be the two projections.

If we let $U \subset X \times Z$ to be the universal cycle and $n\colon \widehat{X} \rightarrow U$ be the normalisation, then  we have
\begin{enumerate}
\item a birational morphism $\beta\colon \widehat{X} \rightarrow X$ where $\beta = p_X\vert_{U} \circ n$;

\item 
an equidimensional  contraction
 $f\colon \widehat{X} \rightarrow \widehat{Z}$ where $\widehat{Z}$ is the normalisation of $Z$; and 

\item a foliation $\widehat{\mathcal F} \coloneqq \beta^{-1}\mathcal F$, which is induced by 
$f$.
\end{enumerate}
\end{construction}

\begin{lemma}
\label{lem_neg_disc}
Set up as above. Suppose
that $K_{\mathcal F}$ is $\mathbb Q$-Cartier.  

Then
we may write $K_{\widehat{\mathcal F}}+F \sim_{\mathbb Q} \beta^*K_{\mathcal F}$ where $F \geq 0$
is $\beta$-exceptional.
Moreover, if $E$ is a $\beta$-exceptional divisor which is not $\widehat{\mathcal F}$-invariant
then $E$ is contained in ${\rm Supp}\, F$.
\end{lemma}
\begin{proof}
This is 
\cite[Section 3.6]{Druel21} 
and \cite[Proposition 4.17]{Druel21}.
\end{proof}

\begin{lemma}
\label{lem_comparison_ram}
Let  $f\colon X \rightarrow Z$ be an equidimesional 
contraction between normal varieties
and let $\mathcal F$ and $\mathcal G$ be foliations on $X$ and  $Z$ respectively such that
$\mathcal F = f^{-1}\mathcal G$.  Let $\mathcal H$ be the foliation induced by $X \rightarrow Z$.

Then 
\[
K_{\mathcal H} = (K_{\mathcal F}-f^*K_{\mathcal G})-R(f)_{{\rm n-inv}}
\]
where $R(f)_{{\rm n-inv}}$ denotes the part of the ramification divisor $R(f)$ of $f$ which is not $\cal F$-invariant. 
\end{lemma}
\begin{proof}
The desired equality may be checked away from a subset of codimension at least two and so, without loss of generality,
we may assume that $X, Z, \mathcal H, \mathcal F$ and  $\mathcal G$
are smooth.
We may then apply Formula (2.1) in \cite[\S 2.9]{Druel17} to conclude.
\end{proof}

%

\begin{theorem}
\label{thm_alg_almost_hol}
Let $X$ be a $\mathbb Q$-factorial klt projective variety and let $\mathcal F$ be a foliation on $X$.
Let  $\mathcal H$ be the algebraic part of $\mathcal F$ and let  $\beta\colon \widehat{X} \rightarrow X$ be the morphism guaranteed 
to exist by Construction \ref{c_fam}.

Then $a(E,\mathcal H)\ge a(E,\mathcal F)$ for any $\beta$-exceptional prime divisor $E$ which is not $\beta^{-1}\cal H$-invariant. 
In particular, if $\cal F$ admits canonical singularities, then $\mathcal H$ is induced by an almost holomorphic map.
\end{theorem}

\begin{proof}
Using the same notation as in Construction \ref{c_fam}, note that $\cal H$ is induced by the rational map  $f \circ \beta^{-1} \colon X\dashrightarrow \widehat{Z}$. 
Let $\widehat{\cal F}\coloneqq \beta^{-1}{\cal F}$ and let $\widehat{\cal H}\coloneqq \beta^{-1}{\cal H}$. 
By definition, 
there exists a purely
transcendental foliation $\widehat{\mathcal G}$ on $\widehat{Z}$ such that 
$\widehat{\mathcal F} = \widehat{f}^{-1}\widehat{\mathcal G}$.

We may write
 \[K_{\widehat{\mathcal F}} +F= \beta^*K_{\mathcal F}\]
and
\[K_{\widehat{\mathcal H}} +  G= \beta^*K_{\mathcal H} \]
where $F, G$ are $\beta$-exceptional and by Lemma \ref{lem_neg_disc}, we have that $G\ge 0$.

By Lemma \ref{lem_comparison_ram} we have 
\[
K_{\widehat{\mathcal H}}+{\widehat f}^*K_{\widehat{\mathcal G}}+R(\widehat f)_{{\rm n-inv}} = K_{\widehat {\cal F}},
\] where
$R(\widehat f)_{{\rm n-inv}}$ denotes the part of $R(\widehat f)$ which is not $\widehat{ \cal F}$-invariant. 
Thus, 
\[
  G-F \sim_{\mathbb Q,\beta} {\widehat f}^*K_{\widehat{\mathcal G}}+ R(\widehat f)_{{\rm n-inv}}. 
\]
Notice that $K_{\widehat{\mathcal G}}$
is pseudo-effective since $\widehat{\mathcal G}$ is purely transcendental (see \cite[Corollary 4.13]{CP19}). Let $A$ be an ample divisor on $\widehat{X}$ and let $\Sigma$ be a $\beta$-exceptional prime divisor.  Since $G-F$ is $\beta$-exceptional and ${\widehat f}^*K_{\widehat{\mathcal G}}+ R(\widehat f)_{{\rm n-inv}}$ is pseudo-effective, the negativity lemma (cf. \cite[Lemma 3.39]{KM98}) implies that if $m_\Sigma(G-F)>0$  then $\Sigma$ is contained in the stable base locus of ${\widehat f}^*K_{\widehat{\mathcal G}}+ R(\widehat f)_{{\rm n-inv}}+tA$ for any sufficiently small rational number $t>0$.  In particular, $\Sigma$ is $\widehat{\cal H}$-invariant. 


We deduce from this that for any $\beta$-exceptional prime divisor $E$ which is not $\widehat{\cal H}$-invariant we have 
$a(E, \mathcal H) \ge a(E, \mathcal F)$.

We now prove our second claim. Assume by contradiction that $\cal F$ admits canonical singularities 
but  $f$ is not almost holomorphic. Then  there exists a $\beta$-exceptional 
divisor $E$ which dominates $\overline{Z}$. 
By Lemma \ref{lem_neg_disc}, we have that $E$ is contained in the support of $G$. Thus, $a(E,\mathcal F)\le a(E,\mathcal H)<0$, a contradiction. 
\end{proof}

\begin{corollary}
\label{cor_main}
Let $X$ be a smooth projective variety and let $\mathcal F$ be a foliation on $X$ with  canonical 
singularities. 
Suppose that $-K_{\mathcal F}$ is nef and is not numerically trivial.

Then the algebraic part of $\mathcal F$ is induced by an equidimensional fibration.
\end{corollary}
\begin{proof}
By \cite[Corollary 4.13]{CP19}, the algebraic part $\mathcal H$ of $\mathcal F$ is non-trivial, and by 
Theorem \ref{thm_alg_almost_hol},  $\mathcal H$ is induced by an almost holomorphic map $f\colon X\dashrightarrow Z$.
Following the proof of \cite[Claim 4.3]{MR3935068} we see that $K_{\mathcal F} \equiv K_{\mathcal H}$ and, in particular $-K_{\cal H}$ is nef.
The result then follows by \cite[Corollary 1.4]{Ou21}.
\end{proof}


\bibliography{math.bib}
\bibliographystyle{alpha}

\end{document}